\documentclass[11pt]{article}

\usepackage[T1]{fontenc} 
\usepackage{lmodern}
\usepackage{blindtext}

\usepackage[utf8]{inputenc}
\usepackage{amsfonts} 
\usepackage{dsfont}
\usepackage{amsmath}
\usepackage{amsmath,latexsym}
\usepackage{amsmath,amssymb}
\usepackage{amsmath,amsthm}
\usepackage[english]{babel}
\usepackage{moreverb}
\usepackage[]{graphics}
\usepackage[]{latexsym} 
\usepackage{amscd}
\usepackage{color}
\usepackage{lineno}
\usepackage{ae}
\usepackage{enumerate}
\usepackage{multicol} 

\usepackage[figuresright]{rotating}
\usepackage{bigints}
\graphicspath{{FIGURES/}}
\makeatletter
\newcommand{\xrightharpoonup}[2][]{\ext@arrow 0359\rightharpoonupfill@{#1}{#2}}
\makeatother
\usepackage{hyperref}
\hypersetup{
	colorlinks=true,
	linkcolor=blue,
	filecolor=magenta,      
	urlcolor=cyan,
}

\usepackage{graphics}
\usepackage{graphicx} 
\usepackage[multiple]{footmisc}

\newcommand{\ve}{\varepsilon}  
\newcommand{\ds}{\displaystyle}

\def\R{\mathbb{R}}

\def\P{\mathbb{P}}
\def\E{\mathbb{E}}

\def\N{{\rm I\hspace{-0.50ex}N} }
\def\N{\mathbb{N}}

\newtheorem{lem}{Lemma}[section]
\newtheorem{thm}[lem]{Theorem}
\newtheorem{cor}[lem]{Corollary}
\newtheorem{defi}[lem]{Definition}
\newtheorem{prop}[lem]{Proposition}

\newtheorem{rmq}[lem]{Remark}

\newtheorem{exemple}[lem]{Example}

\title{\bf New approach to greedy vector quantization}
\author{
\textsc{Rancy El Nmeir} \thanks{Sorbonne Universit\'e, Laboratoire de Probabilit\'e, Statistique et Mod\'elisation, Campus Pierre et Marie Curie, case 158, 4, pl. Jussieu, F-75252 Paris Cedex 5, France.} \thanks{Universit\'e Saint-Joseph de Beyrouth, Laboratoire de Math\'ematiques et Applications, Unit\'e de recherche Math\'ematiques et mod\'elisation, B.P. 11-514 Riad El Solh Beyrouth 1107
2050, Liban.}\quad,\quad 
\textsc{Harald Luschgy}~\thanks{Universit\"at Trier, FB IV-Mathematik, D-54286 Trier, Germany.} $\quad$ and $\quad$ 
\textsc{Gilles Pag\`es}\footnotemark[1]
}

\begin{document}
	\date{}
	\maketitle
	\begin{abstract}We extend some rate of convergence results of greedy quantization sequences already investigated in $\cite{LuPa15}$. We show, for a more general class of distributions satisfying a certain control, that the quantization error of these sequences have an $n^{-\frac1d}$ rate of convergence and that the distortion mismatch property is satisfied. We will give some non-asymptotic Pierce type estimates. The recursive character of greedy vector quantization allows some improvements to the algorithm of computation of these sequences and the implementation of a recursive formula to quantization-based numerical integration. Furthermore, we establish further properties of sub-optimality of greedy quantization sequences.
	\end{abstract}
	\bigskip
	%\fi
	%\chapter{New approach to greedy vector quantization}
	%\label{articlegreedy}
	\section{Introduction}
	Let $d \geq 1$, $r \in (0, + \infty)$ and $L^r_{\mathbb{R}^d}(\mathbb{P})$ (or simply  $L^r (\mathbb{P})$) the set of $d$-dimensional random variables $X$ defined on the probability space $(\Omega,\mathcal{A},\mathbb{P})$ such that $\E\| X\|^r<+\infty$ where $\|.\|$ denotes any norm on $\R^d$. We denote $P=\P_X$ the probability distribution of $X$.
	Optimal vector quantization is a technique derived from signal processing, initially devised to optimally discretize a continuous (stationary) signal for its transmission. Originally developed in the $1950s$ (see \cite{gegr98}), it was introduced as a cubature formula for numerical integration in the early $1990s$ (see \cite{Pages98}) and for approximation of conditional expectations in the early $2000s$ for financial applications (see \cite{BaPaPr01,BaPa03}). Its goal is to find the best approximation of a continuous probability distribution by a discrete one, or in other words, the best approximation of a multidimensional random vector $X$ by a random variable $Y$ taking at most a finite number $n$ of values. \\
	Let $\Gamma=\{x_1,\ldots,x_n\}$ be a $d$-dimensional grid of size $n$. The idea is to approximate $X$ by $q(X)$, where $q$ is a Borel function defined on $\mathbb{R}^d$ and having values in $\Gamma$. %The error induced by such approximation is 
	%$$|X-q(X)| \geq \ds \min_{x_i \in \Gamma} |X-x_i|.$$
	If we consider, for $q$, the nearest neighbor projection $\pi_{\Gamma}:\R^d \rightarrow \Gamma$ defined by 
	$$\pi_{\Gamma} (\xi) =\ds \sum_{i=1}^{n} x_i \mathds{1}_{W_i(\Gamma)}(\xi),$$
	where \begin{equation}
	\label{Voronoicells}
	W_{i}(\Gamma)  \subset\{\xi \in \mathbb{R}^d : \|\xi-x_i\|\leq  \min_{j\neq i} \|\xi-x_j\|\}, \qquad i =1,\ldots,n,
	\end{equation} is the Vorono\"i partition induced by $\Gamma$, then the Vorono\"i quantization of $X$ is defined by
	\begin{equation} 
	\widehat{X}^{\Gamma} =\pi_{\Gamma}(X) := \sum_{i=1}^{n}x_i \mathds{1}_{W_{i}(\Gamma)}(X). 
	\end{equation}
	We will denote, most of the times, $\widehat{X}$ instead of $\widehat{X}^\Gamma$ when there is no need for specifications. The $L^r$-quantization error associated to the grid $\Gamma$ is defined, for every $r\in (0,+\infty)$, by 
	\begin{equation} 
	\label{quanterror}
	e_r(\Gamma,X)  =  \|X-\pi_\Gamma(X)\|_r = \|X-\widehat{X}^{\Gamma}\|_r = \left \|\min_{1\leq i \leq n} |X - x_i|\right \|_r
	\end{equation}
	where $\|.\|_r$ denotes the $L^r(\P)$-norm (or quasi-norm if $0<r<1$). Consequently, the optimal quantization problem comes down to finding the grid $\Gamma$ that minimizes this error. 
	%we define the optimal quantization problem as follows
	%\begin{equation}
	%e_{r,n}(X) \; = \; \inf_{\Gamma, \; {\rm card}(\Gamma) \leq n} e_r(\Gamma,X).
	%\end{equation}
	It has been shown (see \cite{GraLu00,Pages15,Pages18}) that this problem admits a solution and that the quantization error converges to $0$ when the size $n$ goes to $+\infty$. The rate of convergence is given by two well known results exposed in the following theorem. 
	\begin{thm} 
		\label{Zadoretpierce}
		\noindent $(a)$ {\rm Zador's Theorem (see \cite{Zador82})} :  
		Let $X \in L_{\mathbb{R}^d}^{r+\eta}(\P)$, $\eta>0$, 
		with distribution $P$ such that 
		$dP(\xi)= \varphi(\xi) d \lambda_d(\xi)+ d\nu(\xi)$. 
		Then, 
		$$\ds\lim_{n\rightarrow +\infty} n^{\frac{1}{d}} e_{r,n}(X) = \tilde{J}_{r,d} \|\varphi\|^{\frac{1}{r}}_{L^{\frac{r}{r+d}}(\lambda_d)}$$
		where $\tilde{J}_{r,d} = \ds \inf_{n \geq 1} n^{\frac{1}{d}} e_{r,n}(U([0,1]^d)) \in (0,+ \infty)$.  \\
		%	In particular, we have
		%$$ \tilde{J}_{1,2} =\sqrt{\frac{2+3\log\sqrt{3}}{3^\frac{7}{4}\sqrt{2}}}, \quad     \tilde{J}_{2,2} = \sqrt{\frac{5}{18\sqrt{3}}} \quad \mbox{and} \quad \tilde{J}_{2,d} \sim \ds \left(\frac{d}{2\pi {\rm e}}\right)^{\frac{1}{2}} \quad  \mbox{as} \quad  d \rightarrow + \infty.$$
		\noindent $(b)$ {\rm Extended Pierce's Lemma (see \cite{LuPa08})}:
		Let $r,\eta >0 $. There exists a constant $\kappa_{d,r,\eta} \in (0,+\infty)$  such that, 
		$$\forall n\geq 1,\quad e_{r,n}(X) \leq \kappa_{d,r,\eta} \sigma_{r+\eta} (X) n^{- \frac{1}{d}}$$
		where, for every $r \in (0, + \infty), \, \sigma_r(X) = \ds \inf_{a \in \mathbb{R}^d} \|X-a\|_r$ is the $L^r$-standard deviation of $X$.
\end{thm}
However, the numerical implementation of multidimensional optimal quantizers requires the computation of grids of size $N \times d$ which becomes too expensive
when $N$ or $d$ increase. Hence, there is a need to provide a sub-optimal solution
to the quantization problem which is easier to handle and whose convergence rate remains similar (or comparable) to that induced by optimal quantizers.
A so-called greedy version of optimal vector quantization has been developed in \cite{LuPa15}. It consists this time in building a {\it sequence} of points $(a_n)_{n \geq 1}$ in $\R^d$ which is recursively optimal step by step, in the sense that it minimizes the $L^r$-quantization error at each iteration. This means that, having the first $n$ points $a^{(n)}=\{a_1,\ldots,a_n\}$ for $n\geq 1$, we add, at the $(n+1)$-th step, the point $a_{n+1}$ solution to
\begin{equation}
\label{greedydef}
\qquad a_{n+1} \in \mbox{argmin}_{\xi \in \R^d} \, e_r(a^{(n)} \cup \{\xi\}, X),
\end{equation}  
noting that $a^{(0)} = \varnothing$, so that {\em $a_1$ is simply an/the $L^r $-median of the distribution $P$ of $X$}. 
The sequence $(a_n)_{n \geq 1}$ is called an $L^r$-optimal greedy quantization sequence for $X$ or its distribution $P$. The idea to design such an optimal sequence, which will hopefully produce  quantizers with a rate-optimal behavior as $n$ goes to infnity, is very natural and may be compared to sequences with low discrepancy in Quasi-Monte Carlo methods when working on the unit cube $[0,1]^d$. 
In fact, such sequences have already been investigated in an $L^1$-setting for compactly supported distributions $P$ as a model of short term experiment planning versus long term experiment planning represented by regular quantization at a given level $n$ (see \cite{BBSS09}) and, then, in \cite{LuPa15} where the authors investigated more deeply this greedy version of vector quantization for $L^r$-random vectors taking values in $\R^d$. They showed that the problem~$(\ref{greedydef})$ admits at least one solution $(a_n)_{n \geq 1}$ when $X$ is an $\R^d$-valued random vector (the existence of such sequences can be proved in Banach spaces but, in this paper, we will only focus on $\R^d$). This sequence may not be unique since greedy quantization depends on the symmetry of the distribution (consider for example the $\mathcal{N}(0,1)$ distribution). However, note that, if the norm $\|.\|$ is strictly convex and $r>1$, then the $L^r$-median is unique. They also showed that the $L^r$-quantization error converges to $0$ when $n$ goes to infinity
%\begin{equation}
%\label{converreur}
%%\lim_{n \rightarrow +\infty} e_r (a^{(n)} , X) = 0,
%\end{equation}
and, if $\mbox{supp}(P)$ contains at least $n$ elements, then the sequence $a^{(n)}$ lies in the convex hull of supp($P$), $e_r (a^{(k)} , X)$ is decreasing w.r.t. $k \in \{1,\ldots,n\}$ and $P\left( \{\xi \in \R^d: \|\xi -a_n\| < \min_{1 \leq i \leq n } \|\xi-a_i\|\} \right) > 0$. The proof of these results (see Propositions $2.1$ and $2.2$ in \cite{LuPa15}) are based on micro-macro inequalities given in \cite{mismatch08}. Moreover, the authors showed in \cite{LuPa15} that these sequences have an optimal rate of convergence to zero, compared to optimal quantizers, and that they satisfy the distortion mismatch problem, i.e. the property that the optimal rate of $L^r$-quantizers holds for $L^s$-quantizers for $s>r$. The proofs were based on the integrability of the $b$-maximal functions associated to an $L^r$-optimal greedy quantization sequence $(a_n)_{n \geq 1}$ given by 
\begin{equation}
\label{psib}
\forall \xi \in \R^d, \quad \Psi_b(\xi)=\sup_{n \in \N}\frac{\lambda_d\left(B(\xi,b \, \mbox{dist}(\xi,a^{(n)}))\right)}{P\left(B(\xi,b \, \mbox{dist}(\xi,a^{(n)}))\right)}.
\end{equation}
In this paper, we will extend those rate of convergence and distortion mismatch results to a much larger class of functions. Instead of maximal functions, we will rely on a new micro-macro inequality involving an auxiliary probability distribution $\nu$ on $\R^d$. When this distribution $\nu$ satisfies an appropriate control on balls, with respect to an $L^r$-median $a_1$ of $P$, defined later in section~$\ref{rateoptimality}$, we will show that the rate of convergence of the $L^r$-quantization error of greedy sequences is $\mathcal{O}(n^{-\frac1d})$, just like the optimal quantizers. Furthermore, considering appropriate auxiliary distributions $\nu$ satisfying this control allows us to obtain Pierce type, and hybrid Zador-Pierce type, $L^r$-rate optimality results of the error quantization, instead of only Zador type results as given in \cite{LuPa15}.\\ %For absolutely continuous distributions $\P=h.\lambda_d$, we will obtain, under some quasi-radial assumption on the density $h$, a hybrid Zador-Pierce theorem by specifying an auxiliary measure involving $h$.\\ %In addition to these results, the distortion mismatch problem is solved in this framework, considering the same probability densities as those used to solve the rate of decay of the error quantization. \\

A very important field of applications is to use these greedy sequences instead of $n$-optimal quantizers in quantization-based numerical integration schemes. In fact, the size of the grids used in these procedures is large in a way that the RAM storing of the quantization tree may exceed the storage capacity of the computing device. So, using greedy quantization sequences will dramatically reduce this drawback, especially since we will show that they behave similarly to optimal quantizers in terms of convergence rate. The computation of greedy quantizers is performed by algorithms, detailed in \cite{LuPa215}, allowing also the computation of the weights $(p_i^n)_{1 \leq i \leq n}$ of the Vorono\"i cells of the sequence $a^{(n)}$. Theses quantities are  mandatory for the greedy quantization-based numerical integration to approximate an integral $I$ of a function $f$ on $\R^d$ by the cubature formula 
$$I(f) \approx \sum_{i=1}^n p_i^n f\big(a_i^{(n)}\big).$$  
Compared to other methods of numerical approximation, such as quasi-Monte Carlo methods (QMC), the quantization-based methods present an advantage in terms of convergence rate, since QMC, for example, is known to induce a convergence rate of $\mathcal{O}\Big(\frac{\log n}{n^{\frac1d}} \Big)$ when integrating Lipschitz functions (see \cite{Proinov88}) while quantization-based numerical integration produces an $\mathcal{O}\big(n^{-\frac1d}\big)$ rate (see \cite{Pages18}). However, it seems to have a drawback which is the computation of the non-uniform weights $(p_i^n)_{1 \leq i \leq n}$, unlike the uniform weights in QMC (equal to $\frac1n$). In this paper, we expose how the recursive character of greedy quantization provides several improvements to the algorithm, making it more advantageous. Moreover, this character induces the implementation of a recursive formula for numerical integration, that can replace the usual cubature formula, reducing the time and cost of the computations. This recursive formula will be introduced first in the one-dimensional case, and then extended to the multi-dimensional case for product greedy quantization sequences, computed from one-dimensional sequences, used to reduce the cost of implementations while always preserving the recursive character.\\

The paper is organized as follows. We first show that greedy quantization sequences can be rate optimal just like the optimal quantizers in section~\ref{rateoptimality} where we extend the results already presented in \cite{LuPa15} and we give Pierce type results. Likewise, the distortion mismatch problem will be solved and extended in section~\ref{distortionmismatch}. In section~\ref{algorithmics}, we present the improvements we can apply to the algorithm of designing the greedy sequences, as well as the new approach for greedy quantization-based numerical integration. Numerical examples will illustrate and confirm the advantages brought by this new approach in section~\ref{numericalexample}. Finally, section~\ref{numericalproperties} is devoted to some numerical conclusions about further properties of greedy quantization sequences such as the sub-optimality, the convergence of empirical measures, the stationarity (or quasi-stationarity) and the discrepancy, to see to what extent greedy sequences can be close to optimality.

\section{Rate optimality: Universal non-asymptotic bounds}
\label{rateoptimality}
In \cite{LuPa15}, the authors presented the rate optimality of $L^r$-greedy quantizers in the sense of Zador's theorem based on the integrability of the $b$-maximal function $\Psi_b(\xi)$ defined by~$(\ref{psib})$. Here, we present Pierce type non-asymptotic estimates relying on micro-macro inequalities applied to a certain class of auxiliary probability distributions $\nu$. Different specifications of $\nu$ lead to various versions of Pierce's Lemma. \\
In all this section, we denote $V_d=\lambda_d\big(B(0,1)\big)$ w.r.t. the norm $\|\cdot\|$.
%\subsection{Universal nonasymptotic bounds and rate optimality}
%Let $(a_n)_{n \geq 1}$ be an $L^r$-optimal greedy sequence for a random vector $X$ with distribution $P$ and taking values in $\R^d$. We note $a^{(n)}=\{a_1,\ldots,a_n\}$. 
We recall, first, a micro-macro inequality that will be be used to prove the first result.
\iffalse
\begin{prop}[Micro-macro inequality]
Let $\Gamma \subset \R^d$ be a finite quantizer of a random variable $X$ with distribution $P$. Then, for every $c \in (0,\frac12)$ and $y \in \R^d$,
\begin{equation}
\label{micromacro}
e_r(\Gamma,P)^r - e_r(\Gamma \cup \{y\},P)^r 
\geq ((1-c)^r-c^r) P \left(B\left(y, cd\left(y,\Gamma\right)  \right) \right) d\left(y,\Gamma\right)^r. 
\end{equation}
\end{prop}
\begin{proof}
	Let $\Gamma_1=\Gamma \cup \{y\}$. We have $B(y, cd(y,\Gamma))\subset W_{y}(\Gamma_1)$, where $ W_{y}(\Gamma_1)$ is the Vorono\"i cell associated to $\Gamma_1$ of centroid $y$, so that for every $x \in B(y, cd(y,\Gamma))$,
	$$d(x,\Gamma) \geq d(y,\Gamma)=\|x-y\| \geq (1-c)d(y,\Gamma).$$
	Consequenty,
	\begin{align*}
	e_r(\Gamma,P)^r - e_r(\Gamma \cup \{y\},P)^r & = \int_{\R^d} \left(d(x,\Gamma)^r -d(x,\Gamma_1)^r\right)dP(x)\\
	& \geq \int_{W_y(\Gamma_1)} \left(d(x,\Gamma)^r-\|x-y\|^r \right)dP(x)\\
	& \geq \int_{B(y, cd(y,\Gamma))} ((1-c)^r-c^r) d(y,\Gamma)^r dP(x)\\
	& =((1-c)^r-c^r)P\left(B(y, cd(y,\Gamma)) \right) d(y,\Gamma)^r.
	\end{align*}
	\hfill $\square$
\end{proof}
\fi
\begin{prop}
	\label{inc}
	Assume $\int\|x\|^rdP(x)<+\infty$. Then, for every probability distribution $\nu$ on $(\R^d,\mathcal{B}(\R^d))$, every $c \in (0, \tfrac 12)$ and every $n \geq 1$
	\begin{align*}
	e_r(a^{(n)},P)^r - e_r(a^{(n+1)},P)^r 
	 \geq \frac{(1-c)^r-c^r}{(c+1)^r} \int \nu \left(B\left(x, \frac{c}{c+1}d\left(x,a^{(n)}\right)  \right) \right) d\left(x,a^{(n)}\right)^r dP(x). 
	\end{align*}
\end{prop}
\begin{proof}
	\noindent \textsc{\bf Step 1:} Micro-macro inequality\\
	Let $\Gamma \subset \R^d$ be a finite quantizer of a random variable $X$ with distribution $P$ and $\Gamma_1=\Gamma \cup \{y\}$, $y \in \R^d$. For every $c \in (0, \tfrac 12)$, we have $B(y, cd(y,\Gamma))\subset W_{y}(\Gamma_1)$, where $ W_{y}(\Gamma_1)$ is the Vorono\"i cell associated
	% to $\Gamma_1$ of 
	centroid $y$ form a Voronoi partition induced by $\Gamma_1$, as defined by~$(\ref{Voronoicells})$. Hence,  for every $x \in B(y, cd(y,\Gamma))$,
	$d(x,\Gamma) \geq d(y,\Gamma)-\|x-y\| \geq (1-c)d(y,\Gamma).$ 
	Consequently,
	\begin{align*}
	e_r(\Gamma,P)^r - e_r(\Gamma \cup \{y\},P)^r & = \int_{\R^d} \left(d(x,\Gamma)^r -d(x,\Gamma_1)^r\right)dP(x)\\
	& \geq \int_{W_y(\Gamma_1)} \left(d(x,\Gamma)^r-\|x-y\|^r \right)dP(x)\\
	& \geq \int_{B(y, cd(y,\Gamma))} ((1-c)^r-c^r) d(y,\Gamma)^r dP(x).
	\end{align*}
	Finally, we obtain the micro-macro inequality
	\begin{equation}
	\label{micromacro}
	e_r(\Gamma,P)^r - e_r(\Gamma \cup \{y\},P)^r 
	\geq ((1-c)^r-c^r) P \left(B\left(y, cd\left(y,\Gamma\right)  \right) \right) d\left(y,\Gamma\right)^r. 
	\end{equation}
	\textsc{\bf Step 2:} 
	Based on the micro-macro inequality~$(\ref{micromacro})$, we have for every $ c\in \left(0, \frac{1}{2} \right)$ and every $ y \in \R^d,$
	$$e_r(a^{(n)},P)^r- e_r(a^{(n)}\cup \{y\},P)^r \geq \left((1-c)^r-c^r\right)P\left(B\left(y,cd(y,a^{(n)})\right)\right) d(y,a^{(n)})^r .$$
	Since $e_r(a^{(n+1)},P)\leq e_r(a^{(n)} \cup \{y\},P)$ for every $y \in \R^d$,
	$$e_r(a^{(n)},P)^r- e_r(a^{(n+1)},P)^r \geq \left((1-c)^r-c^r\right)P\left(B\left(y,cd(y,a^{(n)})\right)\right) d(y,a^{(n)})^r .$$
	We integrate this inequality with respect to $\nu$ to obtain 
	$$e_r(a^{(n)},P)^r- e_r(a^{(n+1)},P)^r  \geq ((1-c)^r-c^r)\int P\left(B\left(y,cd(y,a^{(n)})\right)\right) d(y,a^{(n)})^rd\nu(y).$$
	Now, we consider the closed sets 
	$$F_1=\left\{ (x,y) \in (\R^d)^2: \|x-y\|\leq cd(y,a^{(n)}) \right\} \quad \mbox{and} \quad F_2=\left\{ (x,y) \in (\R^d)^2: \|x-y\|\leq \frac{c}{c+1}d(x,a^{(n)}) \right\}.$$
	We notice that 
	$$F_2 \subset F_1 \cap \left\{(x,y) \in (\R^d)^2: d(y,a^{(n)}) \geq \frac{1}{c+1} d(x,a^{(n)}) \right\},$$
	In fact, for $(x,y) \in F_2$, 
	$$d(y,a^{(n)}) \geq d(x,a^{(n)})-\|x-y\| \geq d(x,a^{(n)})-\frac{c}{c+1}d(x,a^{(n)}) \geq \frac{1}{c+1} d(x,a^{(n)})$$
	and 
	$$\|x-y\| \leq \frac{c}{c+1} d(x,a^{(n)}) \leq c d(y,a^{(n)}).$$
	Then, 
	\begin{align*}
	\int P(B(y,cd(y,a^{(n)}))) d(y,a^{(n)})^rd\nu(y) &=  \int \int \mathds{1}_{F_1}(x,y) d(y,a^{(n)})^r d\nu(y) dP(x) \\
	& \geq \frac{1}{(c+1)^r} \int \int \mathds{1}_{F_2}(x,y) d(x,a^{(n)})^r d\nu(y) dP(x) \\
	& =\frac{1}{(c+1)^r} \int \nu \left(B\left(x, \frac{c}{c+1}d\left(x,a^{(n)}\right)  \right) \right) d(x,a^{(n)})^r dP(x).
	\end{align*}
	\hfill $\square$
\end{proof}

In order to prove the rate optimality of the greedy quantization sequences and obtain a non-asymptotic Pierce type result, we will consider auxiliary probability distributions $\nu$ satisfying the following control on balls with respect to an $L^r$-median $a_1$ of $P$: for every $\varepsilon \in (0,\ve_0)$, for some $\ve_0\!\in (0, 1]$, there exists a Borel function $g_{\varepsilon}: \R^d \rightarrow [0,+\infty)$ such that, for every $x \in \mbox{supp}(P)$ and every $t \in [0,\varepsilon \|x-a_1\|]$, 
\begin{equation}
\label{criterenu}
  \nu(B(x,t)) \geq g_{\varepsilon}(x) V_d t^d.
\end{equation}
Of course, this condition is of interest only if the set $\{g_\ve >0\}$ is sufficiently large. Note that $a_1 \in a^{(n)}$ for every $n \geq 1$ by construction of the greedy quantization sequence. We begin by a technical lemma which will be used in the proof of the next proposition.
\begin{lem}
	\label{lemmab31}
	Let $C, \rho \in (0,+\infty)$ be some real constants and $(x_n)_{n \geq 1}$ be a non-negative sequence satisfying, for every $n \geq 1$,
	\begin{equation}
	\label{lemb31}
		x_{n+1} \leq x_n -Cx_n^{1+\rho}.
	\end{equation}
	 Then for every $n \geq 1$,
	$$(n-1)^{\frac{1}{\rho}}x_n \leq \left(\frac{1}{C\rho} \right)^{\frac{1}{\rho}}.$$
\end{lem}
\begin{proof}
	We rely on the following Bernoulli inequalities, for every $x \geq -1$,
	$$
	(1+x)^{\rho} \geq 1+\rho x, \quad \mbox{ if } \rho \geq 1, \qquad \mbox{ and } \qquad (1+x)^{\rho} \leq 1+\rho x, \quad \mbox{ if } 0 < \rho < 1.
	$$
	These inequalities can be obtained by studying the function $f$ defined for every $ x \in (-1,+\infty)$ by $f(x)= (1+x)^{\rho} - (1+\rho x) $. Assuming that $(x_n)_{n \geq 1}$ is non-increasing and that $x_n >0$ for every $n \geq 1$, it follows from~$(\ref{lemb31})$ that
	$$
	\frac{1}{x_{n+1}^{\rho}} \geq \frac{1}{x_n^{\rho}(1-C\, x_n^{\rho})^{\rho}} \geq \frac{1}{x_n^{\rho}}(1+C\, x_n^{\rho})^{\rho}.
	$$
If $\rho \geq 1$, the Bernoulli inequalities imply 
$\frac{1}{x_{n+1}^{\rho}} \geq \frac{1}{x_n^{\rho}}(1+C\, \rho \, x_n^{\rho})=\frac{1}{x_n^{\rho}} +C \rho.$ By induction, one obtains
$$
\frac{1}{x_{n}^{\rho}} \geq \frac{1}{x_{1}^{\rho}} +(n-1)C \rho \geq (n-1)C \rho
$$
to deduce the result easily. If $0 < \rho< 1$, then $-C \rho x_n^{\rho} \geq -1$ for every $n \geq 1$, and the result is deduced by using the Bernoulli inequality and then reasoning by induction. 
\hfill $\square$	
\end{proof}
\begin{prop}
	\label{mainthm}
	Let $P$ be such that $\int_{\R^d} \|x\|^rdP(x)<+\infty$. For any distribution $\nu$ and Borel function $g_{\varepsilon} : \R^d\rightarrow \R_+$, $\varepsilon \in (0,\frac13)$, satisfying~$(\ref{criterenu})$,
	\begin{equation}
	\label{gepsilongeneral}
	\forall n \geq 2, \quad e_r(a^{(n)},P) \leq \varphi_r(\varepsilon)^{-\frac1d} V_d^{-\frac1d} \left(\frac rd \right)^{\frac1d}\left(\int g_{\varepsilon}^{-\frac rd} dP  \right)^{\frac1r} (n-1)^{-\frac1d}
	\end{equation}
	where $\ds \varphi_r(u)=\left(\frac{1}{3^r}-u^r\right)u^d$.
	%where $\ds C_{r,d}=3^{1+\frac rd}\, \left(\frac rd \right)^{\frac1d} \,V_d^{-\frac1d}\left(1+\frac dr \right)^{\frac 1d}\left(1+\frac rd \right)^{\frac 1r}$.
\end{prop}
\begin{proof}
	We may assume that $\int g_\ve^{-\frac rd}dP<+\infty$. Assume $c \in (0 , \frac{\varepsilon}{1-\varepsilon}] \cap (0,\frac12)$ so that $\frac{c}{c+1} \leq \varepsilon$. Moreover $d(x,a^{(n)}) \leq d(x,a_1)$ since $a_1 \in a^{(n)}$. Consequently, for any such $c$, $\ds \frac{c}{c+1} \, d(x,a^{(n)}) \leq \varepsilon \|x-a_1\|$ so that, by~$(\ref{criterenu})$, there exists a function $g_{\varepsilon}$ such that 
	$$\nu \left(B\left(x, \frac{c}{c+1} \,d\big(x,a^{(n)}\big)  \right) \right) \geq V_d \, \left(\frac{c}{c+1} \right)^{d} d(x,a^{(n)})^d\, g_{\varepsilon}(x).$$
	Then, noting that $ \frac{(1-c)^r-c^r}{(1+c)^r} \geq \frac{1}{3^r} -\big( \frac{c}{c+1}\big)^r \; >0$, since $c \in (0,\tfrac 12)$, Proposition~$\ref{inc}$ implies that
	\begin{equation}
	\label{equ1}
	e_r(a^{(n)},P)^r - e_r(a^{(n+1)},P)^r  
	\geq V_d \, \varphi_r\left(\frac{c}{c+1} \right) \int g_{\varepsilon}(x) d(x,a^{(n)})^{d+r} dP(x)
	\end{equation}  
	where $\ds \varphi_r(u)=\left(\frac{1}{3^r}-u^r\right)u^d, \; u \in (0,\tfrac13)$.
	Applying the reverse H\"older inequality with the conjugate H\"older exponents $p =-\frac rd$ and $q=\frac{r}{r+d}$ yields
	\begin{align*}
		e_r(a^{(n)},P)^r - e_r(a^{(n+1)},P)^r 
		 & \geq V_d \, \varphi_r\left(\frac{c}{c+1} \right) \left(\int g_{\varepsilon }(x)^{-\frac rd} dP(x) \right)^{-\frac dr} \left(\int d(x,a^{(n)})^r dP(x) \right)^{1+\frac dr} \\
		 & \geq V_d \, \varphi_r\left(\frac{c}{c+1} \right) \left(\int g_{\varepsilon} (x)^{-\frac rd} dP(x) \right)^{-\frac dr} \left(e_r(a^{(n)},P)^r \right)^{1+\frac dr} .
	\end{align*}
	Then, applying lemma~\ref{lemmab31} to the sequence $x_n=e_r(a^{(n)},P)^r $ with $C=V_d \, \varphi_r\left(\frac{c}{c+1} \right) \, \left(\int g_{\varepsilon} (x)^{-\frac rd} dP(x) \right)^{-\frac dr}$ and $\rho=\frac dr$, one obtains, for every $c \in (0,\tfrac12)$,
	$$e_r(a^{(n)},\P) \leq V_d^{-\frac 1d} \left(\frac rd\right)^{\frac 1d} \varphi_r\left(\frac{c}{c+1} \right)^{-\frac 1d} \left(\int g_{\varepsilon}^{-\frac rd} dP  \right)^{\frac1r} (n-1)^{-\frac1d}.$$
	Since in most applications $\varepsilon \mapsto \left(\int g_{\varepsilon}^{-\frac rd} dP  \right)^{\frac1r}$ is increasing on $(0,1/3)$, we are led to study $\varphi_r\left( \frac{c}{c+1}\right)^{-\frac 1d}$ subject to the constraint $c \in \big(0,\frac{\varepsilon}{1-\varepsilon}\big] \cap \big(0,\frac12\big)$.
	$\varphi_r$ is increasing in the neighborhood of $0$ and $\varphi_r(0)=0$, so, one has, for every $\varepsilon \in (0,\frac 13)$ small enough,
	$\varphi_r\left( \frac{c}{c+1}\right) \leq \varphi_r(\varepsilon), \mbox{ for } c \in (0,\tfrac{\varepsilon}{1-\varepsilon}].$ 
	This leads to specify $c$ as  
	$ c=\frac{\varepsilon}{1-\varepsilon} \mbox{, so that}  \frac{c}{c+1}=\varepsilon,$ 
	 to finally deduce the result.
	 %, for every $\varepsilon \in (0,\frac 13)$,
	%$$\forall n \geq 2, \quad e_r(a^{(n)},P) \leq \varphi_r(\varepsilon)^{-\frac1d} V_d^{-\frac1d} \left(\frac rd \right)^{\frac1d}\left(\int g_{\varepsilon}^{-\frac rd} dP  \right)^{\frac1r} (n-1)^{-\frac1d}$$
%	Since in most applications $\varepsilon \mapsto \left(\int g_{\varepsilon}^{-\frac rd} dP  \right)^{\frac1r}$ is increasing on $(0,1)$, we are led to minimize $\varphi_r\left( \frac{c}{c+1}\right)^{-\frac 1d}$ subject to the constraint $c \in \left(0,\frac{\varepsilon}{1-\varepsilon}\right] \cap \left(0,\frac12\right)$ with an as small value as possible for $\varepsilon$. \\
%	
%	At this stage, noting that $\varphi_r$ is increasing on $\left( 0, \frac 13 \left( \frac{d}{d+r}\right)^{\frac 1r} \right)$ and decreasing on $\left(\frac 13 \left( \frac{d}{d+r}\right)^{\frac 1r},\frac13 \right)$, one concludes that
%	the optimal choice is to set 
%	$$ c=\frac{\varepsilon}{1-\varepsilon} \qquad \mbox{and} \qquad \varepsilon=\frac 13 \left( \frac {d}{r+d} \right)^{\frac 1r}.$$ 
%	which yields
%	\begin{align*}
%	\inf_{c \in (0,\frac{\varepsilon}{1-\varepsilon} \wedge \frac 12)} \varphi_r \left( \frac {c}{c+1} \right)^{-\frac 1d} & =\varphi_r \left( \frac 13 \left( \frac {d}{r+d} \right)^{\frac 1r} \right)^{-\frac 1d}\\
%	& = \left[\left(\frac{1}{3^r}-\frac{1}{3^r}\frac{d}{d+r}\right) \frac{1}{3^d}\left( \frac{d}{d+r} \right)^{\frac dr} \right]^{-\frac 1d}\\
%%	& = 3^{\frac rd+1} \left(1+\frac dr \right)^{\frac 1d}\left(1+\frac rd \right)^{\frac 1r}.
%	\end{align*}
	\hfill $\square$\\
\end{proof}

By specifying the measure $\nu$ and the function $g_{\varepsilon}$, we will obtain two first natural versions of the Pierce Lemma.

\begin{thm}[Pierce's Lemma]
	\label{Piercetype}
	$(a)$ Assume $\int_{\R^d} \|x\|^r dP(x) < +\infty$. Let $\delta >0$. Then $e_r(a_1,P)=\sigma_r(P)$ and
		$$\forall n \geq 2, \qquad e_r(a^{(n)},P) \leq \kappa_{d,\delta,r}^{\text{Greedy,Pierce}} \sigma_{r+\delta}(P) (n-1)^{-\frac1d}$$
		where $\ds \kappa_{d,\delta,r}^{\text{Greedy, Pierce}} \leq V_d^{-\frac1d} \left(\frac rd \right)^{\frac1d} \left(\left(\frac{\delta}{r} \right)^{\frac{r}{r+\delta}}+\left(\frac{r}{\delta}\right)^{\frac{\delta}{r+\delta}} \right)^{1+\tfrac{\delta}{r}} \left( \int_{\R^d} (\|x\| \vee 1)^{-d-\frac{d\delta}{r}}dx \right)^{\frac 1d}\, \min_{\varepsilon \in (0,\tfrac 13)} (1+\varepsilon)\varphi_r(\varepsilon)^{-\frac 1d}$.\\
		$(b)$ Assume $\int_{\R^d} \|x\|^r dP(x) < +\infty$. Let $\delta >0$. Then
		$$\forall n \geq 2, \; e_r(a^{(n)}, P) \leq \kappa_{d,r,\delta}^{\text{Greedy}} \left(\int \left( \|x-a_1\| \vee 1 \right) ^r \left( \log (\|x-a_1\| \vee e) \right)^{\frac{r}{d}+\delta} dP(x) \right)^{\frac{1}{r}}(n-1) ^{-\frac{1}{d}}$$ 
		where $
		\kappa_{d,r,\delta}^{\text{Greedy}} \leq  V_d^{-\frac1d}\left(\frac rd \right)^{\frac 1d}  \min_{\varepsilon \in (0,\frac13)} (1+\varepsilon) \varepsilon^{\frac 1d +\frac{\delta}{r}} \varphi_r(\varepsilon)^{-\frac 1d}.   \left(\int \frac{dx}{(1 \vee \|x\|)^{d}\left(\log(\|x\| \vee e) \right)^{1+\frac{d\delta}{r}}} \right)^{\frac 1d}.$\\
		In particular, if 
		$\int_{\R^d} \|x\|^r (\log^+\! \|x\|)^{\frac{r}{d}+\delta} dP(x) < +\infty$, then
		$$\lim \sup_n n^{\frac{1}{d}}\sup \{e_r(a^{(n)},P): (a_n) \, L^r\mbox{-optimal greedy sequence for } P \} < +\infty.$$
\end{thm}
\begin{proof} 
	$(a)$ Let $\delta>0$ be fixed. We set $\nu(dx)=\gamma_{r,\delta}(x) \lambda_d(dx)$ where
		$$\gamma_{r,\delta}(x)=\frac{K_{\delta,r}}{(1 \vee \|x-a_1\|)^{d(1+\frac {\delta}{r})}}\;  \qquad \mbox{with} \qquad  K_{\delta,r}= \left(\int \frac{dx}{(1 \vee \|x\|)^{d(1+\frac {\delta}{r})}} \right)^{-1} <+\infty$$
		is a probability density with respect to the Lebesgue measure on $\R^d$.\\
		Let $\varepsilon \in (0,1)$ and $t >0$. For every $x \in \R^d$ such that $\varepsilon \|x-a_1\| \geq t $ and every $y \in B(x,t)$,
		$\|y-a_1\| \leq \|y-x\|+\|x-a_1\| \leq (1+\varepsilon) \|x-a_1\|$ so that 
		$$\nu(B(x,t)) \geq \frac{K_{\delta,r} V_d \,t^d}{\left(1 \vee \left[(1+\varepsilon)\|x-a_1\|\right]\right)^{d(1+\frac {\delta}{r})}}.$$
		Hence,~$(\ref{criterenu})$ is verified with $$g_{\varepsilon}(x)=\frac{K_{\delta,r}}{\left(1 \vee \left[(1+\varepsilon)\|x-a_1\|\right]\right)^{d(1+\frac {\delta}{r})}},$$
		so we can apply Proposition~$\ref{mainthm}$.  We have
		$$\int g_{\varepsilon}(x)^{-\frac rd} dP(x) \leq K_{\delta,r}^{-\frac rd} \int \left(1 \vee (1+\varepsilon)\|x-a_1\| \right)^{r+\delta} dP(x)$$
		so that, applying $L^{r+\delta}$-Minkowski inequality, one obtains	
		$$\left(\int g_{\varepsilon}(x)^{-\frac rd} dP(x) \right)^{\frac 1r} \leq K_{\delta,r}^{-\frac 1d} \left(1 + (1+\varepsilon)\sigma_{r+\delta} \right)^{1+\tfrac{\delta}{r}}.$$
		Consequently, by Proposition~$\ref{mainthm}$, for $\ve \!\in (0,1/3)$,
		\begin{align}
		\label{equ1b}
		e_r(a^{(n)},P) \leq  V_d^{-\frac1d} \left(\frac rd \right)^{\frac1d} K_{\delta,r}^{-\frac1d}\left(1 + (1+\varepsilon)\sigma_{r+\delta}\right)^{1+\tfrac{\delta}{r}} \varphi_r(\varepsilon)^{-\frac1d}(n-1)^{-\frac1d}
		\end{align}
		
		Now, we introduce an equivariance argument. For $\lambda>0$, let $X_{\lambda}:=\lambda(X-a_1)+a_1$ and $(a_{\lambda,n})_{n\geq 1}:=(\lambda(a_n-a_1)+a_1)_{n \geq 1}$. It is clear that $(a_{\lambda,n})_{n \geq 1}$ is an $L^r$-optimal greedy sequence for $X_{\lambda}$ and $e_r(a^{(n)},X)=\frac{1}{\lambda}e_r(a_{\lambda}^{(n)},X_{\lambda})$. 
		%In fact,
	%	\begin{equation*}
	%	e_r(a_{\lambda}^{(n)},X_{\lambda})=\E\left[ \min_{1 \leq i \leq n} |\lambda|^r|X-a_i|^r\right]^{\frac{1}{r}}=\lambda e_r(a^{(n)},X).
	%	\end{equation*}
%		\begin{align*}
%		e_r(a_{\lambda}^{(n)},X_{\lambda}) & =\E\left[ \min_{1 \leq i \leq n} |X_{\lambda}-a_{\lambda,i}|^r\right]^{\frac{1}{r}}\\
%		& = \E\left[ \min_{1 \leq i \leq n} |\lambda(X-a_1)+a_1-\lambda(a_i-a_1)-a_1|^r\right]^{\frac{1}{r}}\\
%		&= \E\left[ \min_{1 \leq i \leq n} |\lambda|^r|X-a_i|^r\right]^{\frac{1}{r}}\\
%		&=\lambda e_r(a^{(n)},X).
%		\end{align*}
		Plugging this in inequality~$(\ref{equ1})$ yields 
		\begin{align*}
		e_r(a^{(n)},P) \leq & V_d^{-\frac1d} \left(\frac rd \right)^{\frac1d} K_{\delta,r}^{-\frac1d}\, \frac{1}{\lambda}\, \left(1 + (1+\varepsilon)\,\lambda\,\sigma_{r+\delta}\right)^{1+\tfrac{\delta}{r}} \varphi_r(\varepsilon)^{-\frac1d}(n-1)^{-\frac1d}\\
		\leq & V_d^{-\frac1d} \left(\frac rd \right)^{\frac1d} K_{\delta,r}^{-\frac1d} \left(\lambda^{-\frac{r}{\delta+r}} + (1+\varepsilon)\,\lambda^{\frac{\delta}{\delta+r}}\,\sigma_{r+\delta} \right)^{1+\tfrac{\delta}{r}}\varphi_r(\varepsilon)^{-\frac1d}(n-1)^{-\frac1d}.
		\end{align*}
		Finally, one deduces the result by setting $\ds \lambda=\frac{r}{\delta} \frac{1}{(1+\varepsilon)\sigma_{r+\delta}}$. \\	
		$(b)$ Let $\delta>0$ be fixed. We set $\nu(dx)=\gamma_{r,\delta}(x) \lambda_d(dx)$ where
		$$\gamma_{r,\delta}(x)=\frac{K_{\delta,r}}{(1 \vee \|x-a_1\|)^{d}\left(\log(\|x-a_1\| \vee e) \right)^{1+\frac{d\delta}{r}}},$$ with $K_{\delta,r}= \left(\int \frac{dx}{(1 \vee \|x\|)^{d}\left(\log(\|x\| \vee e) \right)^{1+\frac{d\delta}{r}}} \right)^{-1} <+\infty,$
		is a probability density with respect to the Lebesgue measure on $\R^d$.\\
		
		Let $\varepsilon \in (0,1)$ and $t >0$. For every $x \in \R^d$ such that $\varepsilon \|x-a_1\| \geq t $ and every $y \in B(x,t)$,
		$\|y-a_1\| \leq \|y-x\|+\|x-a_1\| \leq (1+\varepsilon) \|x-a_1\|$ so that 
		\begin{align*}
		\nu(B(x,t)) \geq & \frac{K_{\delta,r}V_d t^d}{(1 \vee (1+\varepsilon)\|x-a_1\|)^{d}\left(\log((1+\varepsilon)\|x-a_1\| \vee e) \right)^{1+\frac{d\delta}{r}}}\\
		\geq & \frac{K_{\delta,r}V_d t^d}{(1+\varepsilon)^{d}\, \varepsilon^{1+\frac{d\delta}{r}}\, (1 \vee \|x-a_1\|)^{d}\left(\log(\|x-a_1\| \vee e) \right)^{1+\frac{d\delta}{r}}}
		\end{align*}
		since $\log(1+\varepsilon) \leq \varepsilon.$ Hence,~$(\ref{criterenu})$ is verified with $$g_{\varepsilon}(x)=\frac{K_{\delta,r}}{(1+\varepsilon)^{d}\, \varepsilon^{1+\frac{d\delta}{r}}\, (1 \vee \|x-a_1\|)^{d}\left(\log(\|x-a_1\| \vee e) \right)^{1+\frac{d\delta}{r}}},$$
		so we can apply proposition~$\ref{mainthm}$.  We have
		$$\left(\int g_{\varepsilon}(x)^{-\frac rd} dP(x)\right)^{\frac 1r} \leq K_{\delta,r}^{-\frac 1d} (1+\varepsilon) \varepsilon^{\frac 1d +\frac{\delta}{r}} \left( \int \left(1 \vee \|x-a_1\| \right)^{r} \left(\log(\|x-a_1\| \vee e) \right)^{\delta+\frac rd}dP(x)\right)^{\frac1r}.$$
		Consequently, one applies Proposition~$\ref{mainthm}$ to deduce the first part.
%		\begin{equation*}
%		e_r(a^{(n)},P) \leq \kappa_{d,r,\delta}^{\text{Greedy}} \left(\int \left(1 \vee \|x-a_1\| \right)^{r} \left(\log(\|x-a_1\| \vee e) \right)^{\delta+\frac rd}dP(x) \right)^{\frac 1r} (n-1)^{-\frac 1d}
%		\end{equation*}
%		where 
%		\begin{align*}
%		\kappa_{d,r,\delta}^{\text{Greedy}} & \leq V_d^{-\frac1d}\left(\frac rd \right)^{\frac 1d} K_{\delta,r}^{-\frac 1d} \min_{\varepsilon \in (0,\frac13)} (1+\varepsilon) \varepsilon^{\frac 1d +\frac{\delta}{r}} \varphi_r(\varepsilon)^{-\frac 1d}.
%		We chose $\ds \varepsilon=\frac13\left(\frac{d}{d+r} \right)^{\frac 1r}$ as in the proof of Proposition~$\ref{mainthm}$. Consequently, 
%		\begin{equation*}
%		e_r(a^{(n)},P) \leq \kappa_{d,r,\delta}^{\text{Greedy}} \left(\int \left(1 \vee \|x-a_1\| \right)^{r} \left(\log(\|x-a_1\| \vee e) \right)^{\delta+\frac rd}dP(x) \right)^{\frac 1r} (n-1)^{-\frac 1d}
%		\end{equation*}
%		where 
%		\begin{align*}
%		\kappa_{d,r,\delta}^{\text{Greedy}} & \leq C_{r,d} K_{\delta,r}^{-\frac 1d} \left(1+\frac13\left(\frac{d}{d+r} \right)^{\frac 1r} \right) \left(\frac13\left(\frac{d}{d+r} \right)^{\frac 1r} \right)^{\frac 1d +\frac{\delta}{r}} 
%		\end{align*}
		For the second part of the proposition, we  start by noticing that 
		$$\left(1 \vee \|x-a_1\| \right)^{r} \leq \left(1 + \|x\|+\|a_1\| \right)^{r}\leq 2^{(r-1)_+}\left(\|x\|^r+(1+\|a_1\|)^r \right)$$
		and 
		\begin{equation*}
		\log(\|x-a_1\| \vee e) \leq \log(\|x\| \vee e)+ \frac{\|a_1\|\vee e}{\|x\| \vee e} \leq \log_+ \|x\|+1+\frac{\|a_1\|\vee e}{e}
		\end{equation*}
%		\begin{align*}
%		\log(\|x-a_1\| \vee e) &\leq \log(\|x\| \vee e+\|a_1\|\vee e)\\
%		& \leq \log(\|x\| \vee e)+\log \left(1+\frac{\|a_1\|\vee e}{\|x\| \vee e}\right)\\
%		& \leq \log(\|x\| \vee e)+ \frac{\|a_1\|\vee e}{\|x\| \vee e}\\
%		& \leq (1 \vee \log\|x\|)+\frac{\|a_1\|\vee e}{e}\\
%		& \leq \log_+ \|x\|+1+\frac{\|a_1\|\vee e}{e}
%		\end{align*}
		where $\log_+ u =\log u \mathds{1}_{u\geq 1}$, 
		so that
		\begin{align*}
		\left(1 \vee \|x-a_1\| \right)^{r} \left(\log(\|x-a_1\| \vee e) \right)^{\delta+\frac rd}  
		%\leq & 2^{(r-1)_++(\frac rd+\delta-1)_+} \left(\|x\|^r+A_1 \right)\left(\log_+\|x\|^{\frac rd+\delta}+A_2 \right)\\
		\leq  & \; 2^{(r-1)_++(\frac rd+\delta-1)_+} \left( \|x\|^r \log_+\|x\|^{\frac rd+\delta}+A_2\|x\|^r\right.\\
		& \left. +A_1\log_+\|x\|^{\frac rd+\delta}+A_1A_2 \right)
		\end{align*}
		where $A_1= (1+\|a_1\|)^r$ and $A_2=\left(1+\frac{\|a_1\|\vee e}{e}\right)^{\frac rd+\delta}$.
		Since $\log\|x\|^{\frac rd+\delta}=\frac 1r \left(\frac rd+\delta \right) \log\|x\|^r$, then $\log_+\|x\|^{\frac rd+\delta}=\frac 1r \left(\frac rd+\delta \right) \log_+\|x\|^r$. Moreover, $\log_+\|x\|^r \leq \|x\|^r-1$ if $\|x\|^r \geq 1$ and equal to zero otherwise so 
		$$\log_+\|x\|^{\frac rd+\delta} \leq \frac 1r \left(\frac rd+\delta \right) (\|x\|^r-1)_+ \leq \frac 1r \left(\frac rd+\delta \right) (1+\|x\|^r). $$
		Consequently,
		$$\left(1 \vee \|x-a_1\| \right)^{r} \left(\log(\|x-a_1\| \vee e) \right)^{\delta+\frac rd}  \leq 2^{(r-1)_++(\frac rd+\delta-1)_+} \left( \|x\|^r \log_+\|x\|^{\frac rd+\delta}+A'_1\|x\|^r+A'_2 \right)$$
		where 
		$A'_1=A_2+\frac 1r \left(\frac rd+\delta \right)A_2$ and $A'_2=\frac 1r \left(\frac rd+\delta \right)A_1+A_1A_2$. 
		The result is deduced from the fact that $\sup \{\|a_1\|: a_1 \! \in  {\rm argmin}_{\xi\in \R^d}e_r(\{\xi\},P) < +\infty$ (see~\cite[Lemma 2.2]{GraLu00}) and $\kappa_{d,r,\delta}$ does not depend on $a_1$.
	\hfill $\square$
\end{proof}
%\begin{rmq}
%	In Proposition~\ref{Piercetype}(a), one can optimize the upper bound. In fact, instead of considering $\varepsilon=\frac13\left(\frac{d}{d+r} \right)^{\frac 1r}$, one can proceed like in theorem~\ref{mainthm} by minimizing the function 
%	$$\widetilde{\varphi}_r(\varepsilon)=\varphi_r^{-\frac 1d}(\varepsilon)(1+\varepsilon)^{1+\frac{\delta}{r}}=\left[(\frac{1}{3^r}-\varepsilon^r)\varepsilon^d \right]^{-\frac1d} (1+\varepsilon)^{1+\frac{\delta}{r}}.$$
%	Numerical experiments show that this function admits a unique minimum $\varepsilon_{\min}$ depending on the value of $\delta$ that is lower than $\frac13\left(\frac{d}{d+r} \right)^{\frac 1r}$. Thus, by replacing $\varepsilon$ with $\varepsilon_{\min}$, we can find a better constant for the upper bound.
%\end{rmq}
\begin{rmq}
One checks that $\varphi_r$ attains its maximum at $\frac 13 \left( \frac{d}{d+r}\right)^{\frac 1r}$ on $(0,\tfrac13)$, so one concludes that
	%Since $\varphi_r$ is increasing on $\left( 0, \frac 13 \left( \frac{d}{d+r}\right)^{\frac 1r} \right)$ and decreasing on $\left(\frac 13 \left( \frac{d}{d+r}\right)^{\frac 1r},\frac13 \right)$, one concludes that		$$\min_{\varepsilon \in (0,\frac 13)} \varphi_r \left( \varepsilon \right)^{-\frac 1d} =\varphi_r \left( \frac 13 \left( \frac {d}{r+d} \right)^{\frac 1r} \right)^{-\frac 1d}$$
	%so that
	%\begin{equation*}
	$\min_{\varepsilon \in (0,\frac 13)}(1+\varepsilon)\varphi_r \left( \varepsilon \right)^{-\frac 1d} \leq \left(1+\frac13\left(\frac{d}{d+r} \right)^{\frac 1r}\right)3^{\frac rd+1} \left(1+\frac dr \right)^{\frac 1d}\left(1+\frac rd \right)^{\frac 1r}$
%	\end{equation*}
	and \\
	%\begin{equation*}
	$\min_{\varepsilon \in (0,\frac 13)}(1+\varepsilon)\varepsilon^{\frac 1d +\frac{\delta}{r}}\varphi_r \left( \varepsilon \right)^{-\frac 1d} \leq
	\left(1+\frac13\left(\frac{d}{d+r} \right)^{\frac 1r}\right)
	3^{1+\frac{r-1}{d}-\frac{\delta}{r}} \left(1+\frac dr \right)^{\frac 1d-\frac1r}\left(1+\frac rd \right)^{\frac 1r}.$
%	\end{equation*}
	
%	\begin{align*}
%	\min_{\varepsilon \in (0,\frac 13)}(1+\varepsilon)\varphi_r \left( \varepsilon \right)^{-\frac 1d} \leq & \left(1+\frac13\left(\frac{d}{d+r} \right)^{\frac 1r}\right) \varphi_r \left( \frac 13 \left( \frac {d}{r+d} \right)^{\frac 1r} \right)^{-\frac 1d}\\
%	\leq & \left(1+\frac13\left(\frac{d}{d+r} \right)^{\frac 1r}\right)\left[\left(\frac{1}{3^r}-\frac{1}{3^r}\frac{d}{d+r}\right) \frac{1}{3^d}\left( \frac{d}{d+r} \right)^{\frac dr} \right]^{-\frac 1d}\\
%	\leq & \left(1+\frac13\left(\frac{d}{d+r} \right)^{\frac 1r}\right)3^{\frac rd+1} \left(1+\frac dr \right)^{\frac 1d}\left(1+\frac rd \right)^{\frac 1r}.
%	\end{align*}
%	And, on the other hand,
%	\begin{align*}
%	\min_{\varepsilon \in (0,\frac 13)}(1+\varepsilon)\varepsilon^{\frac 1d +\frac{\delta}{r}}\varphi_r \left( \varepsilon \right)^{-\frac 1d} \leq & \left(1+\frac13\left(\frac{d}{d+r} \right)^{\frac 1r}\right) \left(\frac13\left(\frac{d}{d+r} \right)^{\frac 1r} \right)^{\frac 1d +\frac{\delta}{r}} \varphi_r \left( \frac 13 \left( \frac {d}{r+d} \right)^{\frac 1r} \right)^{-\frac 1d}\\
%	\leq & \left(1+\frac13\left(\frac{d}{d+r} \right)^{\frac 1r}\right) \left(\frac13\left(\frac{d}{d+r} \right)^{\frac 1r} \right)^{\frac 1d +\frac{\delta}{r}}  3^{\frac rd+1} \left(1+\frac dr \right)^{\frac 1d}\left(1+\frac rd \right)^{\frac 1r}.
%	\end{align*}
\end{rmq}

At this stage, one can wonder if it is possible to have a kind of hybrid Zador-Pierce result where, if $P=h.\lambda_d$, one has  $$e_r(a^{(n)},P)\leq C \|h\|_{\frac{d}{d+r}} n^{\frac 1d}$$ for some real constant $C$. To this end, we have to consider 
$$\nu=\frac{h^{\frac{d}{d+r}}}{\int h^{\frac{d}{d+r}}d\lambda_d}.\lambda_d.$$
This is related to the following local growth control condition of densities.
\begin{defi}
	Let $A \subset \R^d$. A function $f:\R^d \rightarrow \R_+$ is said to be  almost radial non-increasing  on A w.r.t. $a \in A$ if there exists a norm $\|.\|_0$ on $\R^d$ and real constant
	%s $\eta \geq 0$, 
	$ M\!\in (0,1]$ such that
	\begin{equation}
	\label{radialtails}
	\forall x\! \in 
	%B_{\|.\|_0}(a,\eta)^c \cap 
	A\setminus\{a\}, \quad f_{|B_{\|.\|_0}(a,\|x-a\|_0) \cap (A\setminus\{a\})}
	%\cap B_{\|.\|_0}(a,\eta)^c \cap A } 
	\geq Mf(x).  
	\end{equation}
	If~$(\ref{radialtails})$ holds for $M=1$, then $f$ is called radial non-increasing on $A$ w.r.t. $a$. 
\end{defi}

\begin{rmq}
	\noindent $(a)$~$(\ref{radialtails})$ reads $f(y) \geq M f(x)$ for all $x,y \! \in A\setminus\{a\}$
	%B_{\|.\|_0}(a,\eta)^c \cap A$ 
	for which $\|y-a\|_0 \leq \|x-a\|_0$.
		%\item [$(b)$] If $g,k : \R^d \rightarrow \R_+$ such that $0 < C_1 \leq g \leq C_2 < +\infty$ on $A$ and $k$ having radial non-increasing tails on $A$ w.r.t $a \in A$ with parameters $\|.\|_0$ and $\eta$, then $ f = gk$ has
		%almost radial non-increasing tails on $A$ w.r.t $a$ with parameters $\|.\|_0$, $\eta$ and $M=\frac{C_1}{C_2}$.
		%\item [$(c)$] Let $f$ be a function having almost radial non-increasing tails on $A$ w.r.t. $a \in A$. If $h : \R^d \rightarrow \R_+$ such that $h(x) \sim f(x)$ as $\|x\| \rightarrow \infty$, then $h$ has also almost radial non-increasing tails on $A$ w.r.t. $a \in A$ with the same parameters as $f$.
		
		\smallskip
		\noindent $(b)$ If $f$ is radial non-increasing on $\R^d$ w.r.t. $a \in \R^d$ with parameter $\|.\|_0$,
		%and $\eta$, 
		then there exists a non-increasing measurable function $g:(0, +\infty) \rightarrow \R_+$ satisfying $f(x) = g(\|x -a\|_0 )$ for every $x\neq a$.
		
		\smallskip
		% \in B_{\|.\|_0}(a, \eta)^c$ .\\
		\noindent $(c)$ From a practical point of view, many classes of distributions satisfy~$(\ref{radialtails})$, e.g. the $d$-dimensional normal distribution $\mathcal{N}(m,\sigma_d)$ for which one considers $h(y)=\frac{1}{(2\pi)^{\frac d2}\mbox{\rm det}(\sigma_d)^{\frac12}}e^{-\frac{y^2}{2}}$ and density $f(x)= h(\|x-m\|_0)$ where $\|x\|_0=\|\sigma_d^{-\frac12}x\|$, and the family of distributions defined by $f(x) \propto \|x\|^c e^{-a\|x\|^b}$, for every $x \in \R^d, a,b >0$ and $c>-d$, for which one considers $h(u)=u^ce^{-au^b}$. In the one dimensional case, we can mention the Gamma distribution, the Weibull distributions, the Pareto distributions and the log-normal distributions.
\end{rmq}
\begin{thm}
	\label{zadorpierce}
		Assume $P=h. \lambda_d$ with $h \in L^{\frac{d}{d+r}}(\lambda_d)$ and  $\int_{\R^d} \|x\|^r dP(x) <+\infty$. Let $a_1$ denote the $L^r$-median of $P$. Assume that $\mbox{supp}(P) \subset A$ and $a_1 \in A$ for some $A$ star-shaped and peakless with respect to $a_1$ in the sense that
		\begin{equation}
		\label{peakless}
		\mathfrak{p}(A,\|.-a_1\|) := \inf \left\{ \frac{\lambda_d(B(x,t)\cap A)}{\lambda_d(B(x,t))}; x \in A,0<t\le  \|x-a_1\| \right\} >0.
		\end{equation}
		Assume $h$ is almost radial non-increasing on $A$ with respect to $a_1$ in the sense of~$(\ref{radialtails})$.
		% with $\eta=0$. 
		Then,
		% for every $n \geq 1$, 
		$$ 
		\forall\, n\ge 2,\quad e_r(a^{(n)},P) \leq \kappa_{d,r,M,C_0,\mathfrak{p}(A,\|.-a_1\|)}^{\text{G,Z,P}} \;\|h\|_{L^{\frac{d}{d+r}}(\lambda_d)}^{\frac{1}{r}} (n-1)^{-\frac{1}{d}},
		$$	
		where
		$\kappa_{d,r,M,C_0,\mathfrak{p}(A,\|.-a_1\|)}^{\text{G,Z,P}} \leq  \frac{2C_0^2\,r^{\frac 1d}\,}{d^{\frac1d}M^{d+r}V_d^{\frac1d}\mathfrak{p}(A,\|.-a_1\|)^{\frac1d}} \min_{\varepsilon \in (0,\frac13)} \varphi_r(\varepsilon)^{-\frac 1d}.$
\end{thm}
\begin{rmq}\label{rk:2.9}
$(a)$ If $A=\R^d$, then $\mathfrak{p}(A,\|.-a\|)=1$ for every $a \in \R^d$.

\smallskip
\noindent  $(b)$ The most typical unbounded sets satisfying~$(\ref{peakless})$ are convex cones that is cones  $K \subset \R^d$  of vertex $0$ with $0 \in K$ ($K\neq \varnothing$) and such that $\lambda x \in K$ for every $x \in K$ and $\lambda \geq 0$. For such convex cones $K$  with $\lambda_d(K) >0$, we even have that the lower bound 
		$$
		\mathfrak{p}(K) := \inf \left\{\frac{\lambda_d(B(x,t)\cap K)}{\lambda_d(B(x,t))}; x \!\in K,\,t>0 \right\} = \frac{\lambda_d\big(B(0,1)\cap K) \big)}{V_d}>0.
		$$
		Thus if $K= \R_+^d$, then $\mathfrak{p}(K)= 2^{-d}$.
\end{rmq}
\noindent The proof of theorem~\ref{zadorpierce} is based on the following lemma.
\begin{lem}
	\label{lem2}
	Let $\nu=f . \lambda_d$ be a probability measure on $\R^d$ where $f$ is almost radial non-increasing 
	% (with $\eta=0$) 
	on $A\!\in \mathcal{B}(\R^d)$ w.r.t. $a_1\!\in A$, $A$ being star-shaped relative to $a_1$ and  satisfying~$(\ref{peakless})$.
	% w.r.t. $a_1$. 
	Then, for every $x\! \in A$ and   positive $ t\!\in (0,  \|x-a_1\|]$,
	$$\nu(B(x,t))\geq M \mathfrak{p}(A,\|.-a_1\|)(2C_0^2)^{-d} V_d f(x) t^d$$
	where $C_0 \in [1,+\infty)$ satisfies, for every $x \in \R^d$, $\ds \frac{1}{C_0}\|x\|_0 \leq \|x\| \leq C_0 \|x\|_0$.
	%In particular, this inequality holds for every $x \in \left(B_{\|.\|_0}(a_1, \eta)^c \cup \{a_1\}\right) \cap A$ and every $t \leq \frac{1}{2C_0^2}\|x-a_1\|$.
\end{lem}
\begin{proof}
	For every $x \in A$ and $t > 0$, 
	$$\nu(B(x,t))=\int_{B(x,t)} fd\lambda_d \geq \int_{B(x,t)\cap A \cap \{f \geq Mf(x)\} } fd\lambda_d \;\geq \; Mf(x)\lambda_d\big( B(x,t)\cap A \cap \{f \geq Mf(x)\} \big)$$
	%Since $f$ satisfies~$(\ref{radialtails})$ w.r.t to $a_1$ with parameters $\|.\|_0$ and $\eta=0$, then
	and
	$$B(x,t)\cap (A\setminus\{a_1\}) \cap B_{\|.\|_0}(a_1,\|x-a_1\|_0) \subset B(x,t)\cap A \cap \{f \geq Mf(x)\}.$$
	Now, assume $0<t \le \|x-a_1\|\le C_0\|x-1\|_0$. Setting $x':=\left(1-\frac{t}{2C_0 \|x-a_1\|_0} \right)x+\frac{t}{2C_0\|x-a_1\|_0}a_1 \in A$ (since $A$ is star-shaped with respect to $a_1$), we notice that, for $y \in B\left(x',\frac{t}{2C_0^2} \right) \subset B_{\|.\|_0}\left(x',\frac{t}{2C_0} \right)$, 
	$$
	\|y-x\| \leq \|y-x'\|+C_0\|x'-x\|_0 \leq \frac{t}{2C_0^2}+C_0 \left\|\frac{t}{2C_0\|x-a_1\|_0}(x-a_1) \right\|_0 =\frac{t}{2C_0^2}+\frac{t}{2} \leq t
	$$
	and
	$$
	\|y-a_1\|_0  \leq \|y-x'\|_0+\|x'-a_1\|_0 \leq \frac{t}{2C_0}+ \left\|\left(1-\frac{t}{2C_0\|x-a_1\|_0}\right)(x-a_1) \right\|_0 =\|x-a_1\|_0,
	$$
	so that,
	$B\left(x',\frac{t}{2C_0^2} \right) \subset B(x,t) \cap B_{\|.\|_0}(a_1,\|x-a_1\|_0).$ 
	Consequently,
	$$\nu(B(x,t)) \geq Mf(x) \lambda_d \left(B\left(x',\frac{t}{2C_0^2} \right)\cap A \right).$$
	Moreover, 
	%$x' \in A$ star-shaped w.r.t. $a_1$ and 
	$\frac{t}{2C_0^2}\leq \frac{t}{2} \leq \frac{1}{2} \|x-a_1\| \leq \|x'-a_1\|.$ Hence, we have 
	$$\lambda_d \left(B\left(x',\frac{t}{2C_0^2} \right)\cap A \right) \geq \mathfrak{p}(A,\|.-a_1\|) \lambda_d \left(B\left(x',\frac{t}{2C_0^2} \right) \right) =\mathfrak{p}(A,\|.-a_1\|)(2C_0^2)^{-d}t^d \lambda_d(B(0,1)).$$ 
	\hfill $\square$
\end{proof}
\begin{refproof}
	Consider 
	$$ 
	\nu=h_r .\lambda_d:=\frac{h^{\frac{d}{d+r}}}{\int h^{\frac{d}{d+r}}d\lambda_d}.\lambda_d.
	$$ 
	Notice that $h_r$ is alsmost radial non-increasing on $A$ w.r.t. $a_1$ with parameter $M^{\frac{d}{d+r}}$ so that Lemma~\ref{lem2} yields for every $x\!\in A$ and $t\!\in [0, \|x-a_1\|]$
	\[
	\nu\big( B(x,t)\big)\ge M^{\frac{d}{d+r}}\mathfrak{p}(A,\|\cdot-a_1\|) (2C_0^2)^{-d}V_d h_r(x)t^d.
	\]
	Consequently, using that 
	\[
	\int_{\R^d} h_r^{-\frac rd} dP = \|h\|_{L^{\frac{d}{d+r}}(\lambda_d)},
	\]
	the assertion follows from Proposition~\ref{mainthm}.
	\end{refproof}
\begin{rmq}
	Note that, by applying H\"older inequality with the conjugate exponents $p=1+\frac rd$ and $q =1+\frac dr$, one has
	$$
	\int_{\R^d} h(\xi)^{\frac{d}{d+r}}d\xi \leq \left(\int_{\R^d} h(\xi) (1 \vee |\xi|)^{r+\delta}d\xi \right)^{\frac{d}{d+r}}\left( \int_{\R^d} \frac{d\xi}{(1 \vee |\xi|)^{d(1+\frac{\delta}{r})}}\right)^{\frac{r}{d+r}}.
	$$
	Consequently, since $\ds \int_{\R^d} \frac{d\xi}{(1 \vee |\xi|)^{d(1+\frac{\delta}{r})}} < +\infty$, one deduces that 
	$\|h\|_{\frac{d}{d+r}}^{\frac 1r} \asymp \sigma_{r+\delta}^{1+\frac{\delta}{r}}.$
\end{rmq}
\noindent We note that Zador theorem implies 
$\liminf_n n^{\frac{1}{d}}e_r(a^{(n)},P) \geq \liminf_n n^{\frac{1}{d}} e_{r,n}(P,\R^d)\geq Q_r(P)^{\frac{1}{r}}.$ 
The next proposition may appear as a refinement of Pierce's Lemma and Theorem~$\ref{zadorpierce}$ in the sense that it gives a lower convergence rate for the discrete derivative of the quantization error, that is its increment.
\begin{prop}
	Assume $\int_{\R^d} \|x\|^r dP(x) <+\infty$. Then, 
 	$$\liminf_n n^{1+\frac{r}{d}}\min_{1\leq i \leq n} \left(e_r(a^{(i)},P)^r-e_r(a^{(i+1)},P)^r \right) >0.$$
\end{prop}
\begin{proof}
	We start by choosing $N>0$ such that $P(B(0,N))>0$. Proposition~$\ref{inc}$ yields, for every probability measure $\nu $ on $\R^d$, for every $n \geq n_0$ and $c \in \left(0,\frac{1}{2}\right)$, 
	\begin{align*}
	e_r(a^{(n)},P)^r & - e_r(a^{(n+1)} ,P)^r \\
	& \geq \frac{(1-c)^r-c^r}{(c+1)^r} \int_{B(0,N)\cap supp(P)} \nu \left(B\left(x, \frac{c}{c+1}d\left(x,a^{(n)}\right)  \right) \right) d\left(x,a^{(n)}\right)^r dP.
	\end{align*}
	We choose $\nu=\mathcal{U}(B(0,N))$. Then, for every $x \in B(0,N), \, t \leq N$ and $x'=\left( 1-\frac{t}{2N}\right)x$, one has 
	$B\left( x',\frac{t}{2}\right) \subset B(x,t) \cap B(0,N)$
	since, for every $y \in B\left( x',\frac{t}{2}\right)$,
	$$\|y-x\|\leq \|y-x'\|+\|x'-x\| \leq \frac{t}{2}+\frac{t}{2N}\|x\| \leq t$$
	and $$\|y\|\leq \|y-x'\|+\|x'\|\leq \frac{t}{2}+\left( 1-\frac{t}{2N}\right)\|x\|\leq \frac{t}{2}+\left( 1-\frac{t}{2N}\right)N\leq N.$$
	Consequently, $$\nu(B(x,t)) \geq \frac{\lambda_d(B(x',\frac{t}{2}))}{\lambda_d(B(0,N))}=(2N)^{-d}t^d.$$
	Moreover, we denote $C:=\sup_{n \geq 1} \max_{x \in B(0,N) \cap supp(P)} d(x,a^{(n)})$ which is finite because $a^{(n)} \in \overline{conv}(supp(P))$.
	Consequently, for every $n \geq n_0$ and every $c \in \left(0,\frac{1}{2} \right)$ such that $\ds \frac{c}{c+1}C \leq N$,
	\begin{align*}
	e_r(a^{(n)},P)^r  - e_r(a^{(n+1)} ,P)^r &\geq  \frac{(1-c)^r-c^r}{(c+1)^r} \left(\frac{c}{c+1}\right)^d (2N)^{-d} \int_{B(0,N)} d(x,a^{(n)})^{d+r}dP(x)\\
	& \geq \frac{(1-c)^r-c^r}{(c+1)^r} \left(\frac{c}{c+1}\right)^d (2N)^{-d} P(B(0,N)) e_{d+r}^{d+r}(a^{(n)},P(.|B(0,N))).
	\end{align*}  
	Now, using that $\left(e_{d+r}^{d+r}\left(a^{(n)},P(.|B(0,N))\right)\right)_{n \geq 1}$ is nonincreasing and relying on Zador's theorem, we deduce
	$$\liminf_n n^{1+\frac{r}{d}}\min_{1\leq i \leq n} \left(e_r(a^{(i)},P)^r-e_r(a^{(i+1)},P)^r \right)>0.$$
	\hfill $\square$
\end{proof}
\begin{rmq}
	For every $m,n \in \N$, if we denote $W_b(a^{(n)})$ the Vorono\"i cell associated to the sequence $a^{(n)}$ of centroid $b \in a^{(n)}$ and use the fact that $e_r(a^{(n+1)},X) \leq e_r(a^{(n)}\cup \{b\},X)$ for every $b \in \R^d$, we deduce
	\begin{align*}
	e_r(a^{(n)},X)^r-e_r(a^{(n+m)},X)^r =& \sum_{b \in a^{(n)}} \int_{W_b(a^{(n+m)})} \left( d(x,a^{(n)})^r-\|x-b\|^r \right) dP \\
	&+\sum_{b \in a^{(n+m)}\setminus a^{(n)}} \int_{W_b(a^{(n+m)})} \left( d(x,a^{(n)})^r-\|x-b\|^r \right) dP\\
	& = \sum_{b \in a^{(n+m)}\setminus a^{(n)}} \int_{W_b(a^{(n+m)})} \left( d(x,a^{(n)})^r-d(x,a^{(n)} \cup \{b\})^r \right) dP \\
	%& \leq \sum_{b \in a^{(n+m)}\setminus a^{(n)}} e_r(a^{(n)},X)^r -e_r(a^{(n)}\cup \{b\},X)^r \\
	& \leq m \, e_r(a^{(n)},X)^r-e_r(a^{(n+1)},X)^r.
	\end{align*}
Consequently, considering $n=i$ and knowing that $l \rightarrow e_r(a^{(l)},X)$ is non-increasing, one has
	$$\min_{1 \leq i \leq n}\left(e_r(a^{(i)},X)^r-e_r(a^{(i+1)},X)^r \right) \geq \frac{1}{m} \left(e_r(a^{(n)},X)^r-e_r(a^{(m)},X)^r \right).$$	
\end{rmq}
\section{Distortion mismatch}
\label{distortionmismatch}
We address now the problem of distortion mismatch, i.e. the property that the rate optimal decay property of $L^r$-quantizers remains true for $L^s(P)$-quantization error for $s \in (0, +\infty)$. This problem was originally investigated in \cite{mismatch08} for optimal quantizers. If $s \leq r$, the monotonicity of the $L^s$-norm as a function of $s$ ensures that any $L^r$-optimal greedy sequence remains $L^s$-rate optimal for  the $L^s$-norm. The challenge is when $s$ is larger than $r$. The problem is solved in \cite{LuPa15} for $s \in (0,+\infty)$ relying on an integrability assumption of the $b$-maximal function $\Psi_b$. However, we give an additional nonasymptotic result for $s \in (r,d+r)$ in the following theorem, in the same settings as for Theorem~$\ref{mainthm}$, considering auxiliary probability distributions $\nu$ satisfying~$(\ref{criterenu})$.
\begin{thm}
	\label{mismatch}
	 Let $P$ be such that $\int_{\R^d} \|x\|^rdP(x)<+\infty$. Let $s \in (r,d+r)$.
	 %and assume  $\int \|x\|^r dP(x) <+\infty$. 
	 Let $(a_n)$ be an $L^r$-optimal greedy sequence for $P$. For any distribution $\nu$ and Borel function $g_{\varepsilon} : \R^d\rightarrow \R_+$, $\varepsilon \in (0,\frac13)$, satisfying~$(\ref{criterenu})$, for every $n \geq 3$,
	\begin{equation*}
	e_s\big(a^{(n)},P\big) \leq \kappa_{d,r,\varepsilon}^{\text{Greedy}} \left(\int g_{\varepsilon}^{-\frac{s}{d+r-s}}dP \right)^{\frac{d+r-s}{s(d+r)}}\left(\int g_{\varepsilon}^{-\frac{r}{d}} dP \right)^{\frac{1}{d+r}}(n-2)^{-\frac{1}{d}}
	\end{equation*}
	where
	$
	\kappa_{d,r,\varepsilon}^{\text{Greedy}} = 2^{\frac{1}{d}}\,\,\left(\frac rd \right)^{\frac{r}{d(d+r)}}V_d^{-\frac1d} \varphi_r(\varepsilon)^{-\frac1d}.$
\end{thm} 
\begin{proof}
	We assume $\tfrac{1}{g_\ve} \in L^{\tfrac{s}{d+r-s}}(P)$ so that $\tfrac{1}{g_\ve}\in L^{\frac{r}{d}}(P)$ since $\ds \tfrac{s}{d+r-s} \geq \tfrac{s}{d} \geq \tfrac{r}{d}$. Inequality~$(\ref{equ1})$ from the proof of Proposition~$\ref{mainthm}$ still holds, i.e.
	\begin{equation*}
	e_r(a^{(n)},P)^r  - e_r(a^{(n+1)} ,P)^r 
	\geq C \int g_{\varepsilon}(x) d\left(x,a^{(n)}\right)^{d+r} dP(x). 
	\end{equation*}
	with, for every $c\!\in (0, \tfrac{\ve}{1-\ve}]\cap (0,1/2)$, $C=V_d\, \varphi_r\left(\frac{c}{c+1}\right)$ where $\varphi_r(u)=\left(\tfrac {1}{3^r} -u^r\right)u^d$. The reverse H\"older inequality applied with $p=\frac{s}{d+r} \in (0,1)$ and $q =-\frac{s}{d+r-s} \in (-\infty,0)$ yields that
	%$$\int g_{\varepsilon}(x) d\left(x,a^{(n)}\right)^{d+r} dP(x) \geq \left(\int g_{\varepsilon}^{-\frac{s}{d+r-s}}dP \right)^{-\frac{d+r-s}{s}} \left(\int  d\left(x,a^{(n)}\right)^{s} dP(x) \right)^{\frac{d+r}{s}}$$
	%so that 
	$$e_r(a^{(n)},P)^r - e_r(a^{(n+1)} ,P)^r \geq C_1 e_s\big(a^{(n)},P\big)^{d+r}$$
	where
	$C_1=C\; \big(\int g_{\varepsilon}^{-\frac{s}{d+r-s}}dP \big)^{-\frac{d+r-s}{s}}.$ Hence, knowing that $k \mapsto e_s\big(a^{(k)},P\big)$ is non-increasing and summing between $n$ and $2n-1$, we obtain for $n\geq 1$
	\begin{align*}
	n\, e_s(a^{(2n-1)},P)^{d+r} \leq \sum_{k=n}^{2n-1} e_s\big(a^{(k)},P\big)^{d+r} \leq \frac{1}{C_1} \sum_{k=n}^{2n-1} e_r(a^{(k)},P)^r - e_r(a^{(k+1)} ,P)^r \leq \frac{1}{C_1} e_r(a^{(n)},P)^r.	
	\end{align*}
	Finally, since $ 2 \left\lceil \frac{n}{2} \right\rceil -1 \leq n$, we have $e_s\big(a^{(n)},P\big)\leq  e_s\left(a^{2 \left\lceil \frac{n}{2} \right\rceil -1},P\right)$ and we derive that 
	$$
	\frac{n}{2} e_s\big(a^{(n)},P\big)^{d+r} \leq \left\lceil \frac{n}{2} \right\rceil e_s\big(a^{(n)},P\big)^{d+r} \leq  \left\lceil \frac{n}{2} \right\rceil e_s\left(a^{2 \left\lceil \frac{n}{2} \right\rceil -1},P\right)^{d+r} \leq \frac{1}{C_1}e_r\left(a^{ \left\lceil \frac{n}{2} \right\rceil },P\right)^{r}.
	$$
	Consequently, plugging in $C_1$,
	\begin{align*}
	e_s\big(a^{(n)},P\big)  &\leq  \left(\frac{2}{C_1}\right)^{\frac{1}{d+r}} n^{-\frac{1}{d+r}} e_r\left(a^{ \left\lceil \frac{n}{2} \right\rceil },P\right)^{\frac{r}{d+r}}\\
	 & =  2^{\frac{1}{d+r}} V_d^{-\frac{1}{d+r}} \varphi_r\left(\frac{c}{c+1}\right)^{-\frac{1}{d+r}} \left(\int g_{\varepsilon}^{-\frac{s}{d+r-s}}dP \right)^{\frac{d+r-s}{s(d+r)}}	n^{-\frac{1}{d+r}} e_r\left(a^{ \left\lceil \frac{n}{2} \right\rceil },P\right)^{\frac{r}{d+r}}.
	\end{align*} 
	Consequently, one can deduce from Proposition~$\ref{mainthm}$, for $n\ge 3$,
	\begin{align*}
	e_s\big(a^{(n)},P\big)  \leq &  2^{\frac{1}{d}} V_d^{-\frac{1}{d}} \left(\frac rd\right)^{\frac{r}{d(d+r)}} \left(\int g_{\varepsilon}^{-\frac{s}{d+r-s}}dP \right)^{\frac{d+r-s}{s(d+r)}}\left(\int g_{\varepsilon}^{-\frac rd}dP \right)^{\frac{1}{d+r}} \varphi_r\left(\frac{c}{c+1}\right)^{-\frac{1}{d}} (n-2)^{-\frac1d}.
	%\leq & 2^{\frac{1}{d}} V_d^{-\frac{1}{d}} \left(\frac rd\right)^{\frac{r}{d(d+r)}} \left(\int g_{\varepsilon}^{-\frac{s}{d+r-s}}dP \right)^{\frac{d+r-s}{s(d+r)}}(n-1)^{-\frac{1}{d+r}} \varphi_r\left(\varepsilon\right)^{-\frac{1}{d}} (n-1)^{-\frac1d}
	\end{align*} 
	Hence, the result is owed to the fact that $\varphi_r\left(\frac{c}{c+1}\right) \leq \varphi_r\left(\varepsilon\right)$ for $c \in (0,\frac{\varepsilon}{1-\varepsilon}]$.
	\hfill $\square$
\end{proof}
\begin{cor}
	Let $s \in (r,d+r)$. Assume , for $\delta >0$, 
	\begin{equation}
	\label{assumption}
	\int \|x\|^{\frac{ds}{d+r-s}} (\log^+\! \|x\|)^{\frac{s}{d+r-s}+\delta} dP(x) < +\infty
	\end{equation}
	then
	$$
	\lim \sup_n n^{\frac{1}{d}}\sup \big\{e_s\big(a^{(n)},P\big): (a_n) L^r\mbox{-optimal greedy sequence for } P \big\} < +\infty.
	$$
\end{cor}
\begin{proof} The proof is divided in two steps.

\smallskip
\noindent {\sc Step 1:} Let $\delta>0$ be fixed and $\beta=1+\frac{(d+r-s)\delta}{s}$. Just as in the proof of Theorem~\ref{Piercetype}(b), we set $\nu(dx)=\gamma_{r,\delta}(x) \lambda_d(dx)$ where  
	$$\gamma_{r,\delta}(x)=\frac{K_{\delta,r}}{(1 \vee \|x-a_1\|)^{d}\left(\log(\|x-a_1\| \vee e) \right)^{\beta}}, \mbox{ with }   K_{\delta,r}= \left(\int \frac{dx}{(1 \vee \|x-a_1\|)^{d}\left(\log(\|x-a_1\| \vee e) \right)^{\beta}} \right)^{-1},$$
	is a probability density with respect to the Lebesgue measure on $\R^d$. 
	The density $\gamma_{r,\delta}$ is radial non-increasing on the whole $\R^d$ w.r.t. $a_1$ (and $\|\cdot\|_0= \|\cdot\|$) so that $\mathfrak{p}(\|\cdot-a_1\|)=1$ by Remark~\ref{rk:2.9}$(a)$ and, in turn, Lemma~\ref{lem2} yields for every $x\!\in \R^d$ and $t\le \|x-a_1\|$
	\[
	\nu\big( B(x,t)  \big)\ge 2^{-d} V_d \gamma_{r, \delta}((x) t^d.
	\]
%	Let $\varepsilon \in (0,1)$ and $t >0$. For every $x \in \R^d$ such that $\varepsilon \|x-a_1\| \geq t $ and every $y \in B(x,t)$,
%	$\|y-a_1\| \leq \|y-x\|+\|x-a_1\| \leq (1+\varepsilon) \|x-a_1\|$ so that 
%	\begin{align*}
%	\nu(B(x,t)) \geq &\frac{K_{\delta,r}V_d t^d}{(1 \vee (1+\varepsilon)\|x-a_1\|)^{d}\left(\log((1+\varepsilon)\|x-a_1\| \vee e) \right)^{\beta}}\\
%	\geq & \frac{K_{\delta,r}V_d t^d}{(1+\varepsilon)^{d}\, \varepsilon^{\beta}\, (1 \vee \|x-a_1\|)^{d}\left(\log(\|x-a_1\| \vee e) \right)^{\beta}}
%	\end{align*}
%	since $\log(1+\varepsilon) \leq \varepsilon.$ Hence,~$(\ref{criterenu})$ is verified with $$g_{\varepsilon}(x)=\frac{K_{\delta,r}}{(1+\varepsilon)^{d}\, \varepsilon^{\beta}\, (1 \vee \|x-a_1\|)^{d}\left(\log(\|x-a_1\| \vee e) \right)^{\beta}},$$
%	so we can apply proposition~$\ref{mismatch}$.  We have
%	$$\left(\int g_{\varepsilon}(x)^{-\frac rd} dP(x)\right)^{\frac{1}{r+d}} \leq K_{\delta,r}^{-\frac {r}{d(d+r)}} (1+\varepsilon)^{\frac{r}{d(d+r)}} \varepsilon^{\beta \frac {r}{d(d+r)}} \left( \int \left(1 \vee \|x-a_1\| \right)^{r} \left(\log(\|x-a_1\| \vee e) \right)^{\beta \frac rd}dP(x)\right)^{\frac {1}{d+r}}$$
%	and
%	\begin{align*}
%	\left(\int g_{\varepsilon}(x)^{-\frac{s}{d+r-s}} dP(x)\right)^{\frac{d+r-s}{s(d+r)}} \leq & K_{\delta,r}^{-\frac {1}{d+r}} (1+\varepsilon)^{\frac{d}{d+r}} \varepsilon^{\frac {\beta}{d+r}}\\
%	&\times \left( \int \left(1 \vee \|x-a_1\| \right)^{\frac{sd}{d+r-s}} \left(\log(\|x-a_1\| \vee e) \right)^{\frac{\beta s}{d+r-s}}dP(x)\right)^{\frac {d+r-s}{s(d+r)}}
%	\end{align*}
	Consequently, Theorem~$\ref{mismatch}$ yields, for $n\ge 3$,
	\begin{align*}
	e_s\big(a^{(n)},P\big) \leq & C_{d,r,\delta} \left( \int \left(1 \vee \|x-a_1\| \right)^{r} \left(\log(\|x-a_1\| \vee e) \right)^{\beta \frac rd}dP(x)\right)^{\frac {1}{d+r}} \\
	& \times \left( \int \left(1 \vee \|x-a_1\| \right)^{\frac{sd}{d+r-s}} \left(\log(\|x-a_1\| \vee e) \right)^{\delta+\frac{ s}{d+r-s}}dP(x)\right)^{\frac {d+r-s}{s(d+r)}} (n-2)^{-\frac 1d}
	\end{align*}
	where 
	$
	C_{d,r,\delta}  \leq 2^{1+\frac 1d} V_d^{-\frac1d} \left( \frac rd \right)^{\frac{r}{d(d+r)}}K_{\delta,r}^{-\frac 1d} \min_{\varepsilon \in (0,\frac 13)} (1+\varepsilon)^d \varepsilon^{\frac{\beta}{d}} \varphi_r(\varepsilon)^{-\frac 1d}.
	$\\
	\textsc{Step 2:}
	Just as in the proof of Theorem~\ref{Piercetype}(b), we have 
	$$\left(1 \vee \|x-a_1\| \right)^{r} \left(\log(\|x-a_1\| \vee e) \right)^{\beta \frac rd}  \leq 2^{(r-1)_++(\beta \frac rd-1)_+} \left( \|x\|^r \log_+\|x\|^{\beta \frac rd}+A_1\|x\|^r+A_2 \right)$$
	and
	\begin{align*}
	\left(1 \vee \|x-a_1\| \right)^{\frac{sd}{d+r-s}} \left(\log(\|x-a_1\| \vee e) \right)^{\delta+\frac{ s}{d+r-s}} \leq &2^{(\frac{ds}{d+r-s}-1)_++(\delta+ \frac{s}{d+r-s}-1)_+} \\
	& \times \left( \|x\|^{\frac{ds}{d+r-s}} \log_+\|x\|^{\delta+\frac{s}{d+r-s}}+B_1\|x\|^r+B_2 \right)
	\end{align*}
	where 
	$A_1,A_2,B_1$ and $B_2$ are constants depending only on $r,d,s,\delta$ and $a_1$. 
	Since, $\frac{s}{d+r-s} \geq \frac rd$, one has $\frac{ds}{d+r-s}>r$ and $\delta+\frac{s}{d+r-s}\geq \beta \frac rd$, so that the two above quantities are finite (by assumption~$(\ref{assumption})$).
	The result is deduced from the fact that $\sup \big\{\|a_1\|: a_1\! \in {\rm argmin}_{\xi\in \R^d} e_r(\{\xi\}, P) \big\} < \infty$. \hfill $\square$
\end{proof}
\section{Algorithmics}
\label{algorithmics}
An important application of quantization is numerical integration. Let us consider the quadratic case $r=2$ and an $L^2$-optimal greedy quantization sequence $a^{(n)}$ for a random variable $X$ with distribution $\P_X=P$. Since we know that $e_2(a^{(n)},X)=\|X-a^{(n)}\|_2$ converges to $0$ when $n$ goes to infinity, this means that $a^{(n)}$ converges towards $X$ in $L^2$ and hence in distribution. So, denoting $\big(W_i(a^{(n)})\big)_{1 \leq i \leq n}$ the Vorono\"i diagram corresponding to $a^{(n)}$, one can approximate $\E[f(X)]$, for every continuous function $f:\R^d \rightarrow \R$, by the following cubature formula
\begin{equation}
\label{quadratureformula}
I(f):=\E[f(X)]\approx \sum_{i=1}^{n}p_i^n \, f(a_i^{(n)})
\end{equation}
where, for every $i \in \{1,\ldots,n\}$, $p_i^n=P\big(X \in W_i(a^{(n)}) \big)$ represents the weight of the $i^{th}$ Vorono\"i cell corresponding to the greedy quantization sequence $a^{(n)}=\{a_1^{(n)},\ldots,a_n^{(n)}\}$. When the function $f$ satisfies certain regularities, one can establish error bounds for this quantization-based cubature formula, we refer to \cite{Pages18} for more details.
For example, if $f$ is $[f]_{\text{Lip}}$-Lipschitz continuous, one has 
$$\left|\sum_{i=1}^{n}p_i^n f(a_i^{(n)})-\E[f(X)]\right| \leq [f]_{\text{Lip}} \; e_r(a^{(n)},X).$$
so one can approximate $\E[f(X)]$ with an $\mathcal{O}(n^{-\frac 1d})$ rate of convergence.

When working on the unit cube $[0, 1]^d$, it is natural to compare an optimal greedy sequence of the uniform distribution $\mathcal{U}([0,1]^d)$ and a uniformly distributed sequence with low discrepancy used in the quasi-Monte Carlo method (QMC). A $[0,1]^d$-valued sequence $\xi=(\xi_n)_{n \geq 1}$ is uniformly distributed if $ \mu_n=\frac1n \sum_{k=1}^n\delta_{\xi_k}$ converges weakly to $\lambda_{d_{ \mid [0,1]^d }}$ (where $\lambda_d$ denotes the Lebesgue measure on $(\R^d,\mathcal{B}(\R^d))$). It is well known (see \cite{KuiNied} for example) that $(\xi_n)_{n \geq 1}$ is uniformly distributed if and only if 
$$D_n^*(\xi)=\sup_{u \in [0,1]^d} \left|\frac{1}{n} \sum_{i=1}^{n} \mathds{1}_{\xi_i \in [0,u]^d} -\lambda_d([0,u]^d) \right| \; \rightarrow \; 0 \quad \mbox{ as } \quad n \rightarrow +\infty.$$
The above modulus is known as the star-discrepancy of $\xi$ at order $n$ and can be defined, for fixed $n \in \N$, for any $n$-tuple $(\xi_1,\ldots,\xi_n)$ whose components $\xi_k$ lie in $[0,1]^d$. There exists many sequences (Halton, Kakutani, Faure, Niederreiter, Sobol', see \cite{BoLe94,Pages18} for example) achieving a $\mathcal{O}\left(\frac{(\log n)^d}{n} \right)$ rate of decay for their star-discrepancy and it is a commonly shared conjecture that this rate is optimal, such sequences are called sequences with low discrepancy.  
By a standard so-called Hammersley argument, one shows that if a $[0,1]^{d-1}$-valued sequence $\zeta=(\zeta_n)_{n \geq 1}$ has low discrepancy i.e. there exists a real constant $C(\zeta) \in (0,+\infty)$ such that 
$D_n^*(\zeta) \leq C(\zeta)\frac{(\log n)^{d-1}}{n}$, for every $n \geq 1$,
then, for every $n \geq 1$, the $[0,1]^d$-valued $n$-tuple $\left((\zeta_k,\frac{k}{n}) \right)_{1\leq k \leq n}$ satisfies 
$$D_n^*\left( \left((\zeta_k,\frac{k}{n}) \right)_{1\leq k \leq n}\right) \leq C(\zeta)\frac{(\log n)^{d-1}}{n}.$$
The QMC method finds its gain in the following error bound for numerical integration. Let $(\xi_1,\ldots,\xi_n)$ be a fixed $n$-tuple in $([0,1]^d)^n$, then, for every $f:[0,1]^d \rightarrow \R$ with finite variation (in the Hardy and Krause sense, see \cite{Nied93} or in the measure sense see \cite{BoLe94, Pages18}), 
\begin{equation}
\label{koksmahlawka}
\left|\frac1n \sum_{i=1}^n f(\xi_k) -\int_{[0,1]^d}f(u)du \right| \leq V(f) D_n^*(\xi_1,\ldots,\xi_n).
\end{equation}
where $V(f)$ denotes the (finite) variation of $f$. So, for this class of functions, an $\mathcal{O}\left(\frac{(\log n)^{d-1}}{n} \right)$ or $\mathcal{O}\left(\frac{(\log n)^{d}}{n} \right)$ rate of convergence can be achieved depending on the composition of the sequence. However, the class of functions with finite variation becomes sparser in the space of functions defined from $[0,1]^d$ to $\R$ and it seems natural to evaluate the performance of the low-discrepancy sequences or $n$-tuples on a more natural space of test functions like the Lipschitz functions. This is the purpose of Pro\"inov's theorem reproduced below.
\begin{thm}(Proinov, see \cite{Proinov88})
	\label{proinov}
	Let $(\R^d,\|.\|_{\infty})$. Let $\xi=(\xi_1, \ldots, \xi_n)$ a sequence of $[0,1]^d$. For every continuous function $f:[0,1]^d\rightarrow (\R,|.|_{\infty})$, we define the uniform continuity modulus of $f$ by 
	$w(f, \delta)=\sup_{\xi, \xi^{'} \in [0,1]^d, |\xi -\xi^{'}|_{\infty} \leq \delta} |f(\xi)-f(\xi^{'})|$
	where $|u|_{\infty}=\max_{1 \leq i \leq d} |u_i|$ if $u=(u_1,\ldots,u_d)$.
	Then, for every $n \geq 1$,
	$$\ds \left|  \frac1n \sum_{i=1}^{n}f(\xi_i) - \int_{[0,1]^d}f(x)dx\right| \leq C_d \, w(f, D_n^*(\xi)^{\frac{1}{d}}),$$
	where $C_d$ is a constant lower than $4$ and depending only on the dimension $d$. \\
	In particular, if $f$ is $[f]_{\text{Lip}}-$Lipschitz and $\xi$ has low discrepancy, one has
	$$\ds \left| \frac1n \sum_{i=1}^{n} f(\xi_i) - \int_{[0,1]^d}f(x)dx\right| \leq C_d \,[f]_{Lip} D_n^*(\xi)^{\frac{1}{d}} \leq C_d \,[f]_{Lip} \frac{\log\, n}{n^{\frac{1}{d}}}.$$
\end{thm}
This suggests that, at least for a commonly encountered class of regular functions, the curse of dimensionality is more severe with QMC than with quantization due to the extra $(\log n)^{1-\frac 1d}$ factor in QMC. This is the price paid by QMC for considering uniform weights $p_i=\frac1n, i=1,\ldots,n.$

With greedy quantization sequences, we will show that it is possible to keep the $n^{-\frac1d}$ rate of decay for numerical integration but also keep the asset of a sequence which is a recursive formula for cubatures.
\subsection{Optimization of the algorithm and the numerical integration in the $1$-dimensional case}
Quadratic optimal greedy quantization sequences are obtained by implementing algorithms such as Lloyd's I algorithm, also known as $k$-means algorithm, or the Competitive Learning Vector Quantization (CLVQ) algorithm, which is a stochastic gradient descent algorithm associated to the distortion function. We refer to \cite{LuPa215} (an extended version of \cite{LuPa15} on ArXiv) where greedy variants of these procedures are explained in detail. According to Lloyd's algorithm, the construction of the sequences is recursive in the sense that, at the iteration $n$, we add one point $a_n$ to $\{a_1, \ldots, a_{n-1}\}$, and we denote $\{a_1^{(n)}, \ldots, a_n^{(n)}\}$ an increasing reordering of $\{a_1, \ldots, a_n\}$ where the new added point is denoted by $a_{i_0}^{(n)}$.\\
Since the other points are frozen, we can notice that the local inter-point inertia $\sigma_i^2$  defined by
\begin{equation}
\sigma_i^2:= \ds \int_{a_i^{(n-1)}}^{a_{i+\frac{1}{2}}^{(n-1)}} |a_i^{(n-1)}-\xi|^2 P(d\xi) \, + \,  \int_{a_{i+\frac{1}{2}}^{(n-1)}}^{a_{i+1}^{(n-1)}} |a_{i+1}^{(n)-1}-\xi|^2 P(d\xi), \quad i=0,\ldots, n-1
\end{equation}
(where
$a_0^{(n-1)} =-\infty$,  
$a_{n}^{(n-1)}=+\infty$ and 
$  a_{i+\frac{1}{2}} ^{(n-1)}= \frac{a_i^{(n-1)}+ a_{i+1}^{(n-1)}}{2}$ with $a_{\frac{1}{2}}^{(n-1)}=-\infty ,  \quad a_{n-\frac{1}{2}}^{(n-1)} = +\infty$)
remains untouched for every $ i \in  \{0, \ldots, n-1\}$ except $\sigma_{i_0}^2$ (the inertia between the point $a_{i_0}^{(n)}$ added at the $n$-th iteration and the following point) and $\sigma_{i_0-1}^2$ (the inertia between  $a_{i_0}^{(n)}$  and the preceding point). 
Thus, at each iteration, the computation of $n$ inertia can be reduced to the computation of only $2$, thereby reducing the cost of the procedure. Likewise, the weights $p_i^n=P(W_i(a^{(n)}))$ of the Vorono\"i cells remain mostly unaffected. The only cells that change from one step to another are the cell $W_{i_0}(a^{(n)})$ having for centroid the new point $a_{i_0}^{(n)}$ and the two neighboring cells $W_{i_0-1}(a^{(n)})$ and $W_{i_0+1}(a^{(n)})$. Thus, the online computation of cell weights just needs $3$ calculations instead of $n$ (or $2$ in case the added point is the first or last point in the reordered sequence).
%This way, the construction of the greedy quantization sequence as well as the weights of the Vorono\"i cells will be faster and less expensive, without affecting the calculations relying on these elements. For example, the computation of the quadratic quantization error recursively online  by $$e_2(a^{(n)},X)^2=\int_{\R^d} \min_{1 \leq i \leq n} |a_i^{(n)}-\xi|^2 P(d\xi)=\sum_{i=1}^{n}\int_{W_i(a^{(n)})}  |a_i^{(n)}-\xi|^2 P(d\xi)=\sum_{i=1}^{n} \sigma^2_i$$  
%remains possible and gives the same results as before since the local inertias that have not changed are totally taken into consideration in the sum computed at the previous iteration. \\
The utility of the weights of the Vorono\"i cells is featured in the numerical integration allowing to approximate $\E[f(X)]$ for $f:\R^d \rightarrow \R$ by the quadrature formula~$(\ref{quadratureformula})$ using the reordered sequence $a^{(n)}$. Thus, based on the fact that only $3$ Vorono\"i cells are modified at each iteration, one can deduce an iterative formula for the approximation of $I(f)$ by $I_n(f)$, requiring the storage of only $2$ weights and $2$ indices, as follows
%\begin{eqnarray}
%\label{caliteratifd1}
%\ds I_{n}(f) &=& I_{n-1}(f) - p_{-}^{n}f(a_{i_0-1}^{(n)})-p_{+}^{n}f(a_{i_0+1}^{(n)})+(p_{+}^{n}+p_{-}^{n})f(a_{i_0}^{(n)}) \nonumber \\
%&=& I_{n-1}(f)- p_{-}^{n}(f(a_{i_0-1}^{(n)})-f(a_{i_0}^{(n)}))-p_{+}^{n}(f(a_{i_0+1}^{(n)})-f(a_{i_0}^{(n)})),
%\end{eqnarray}
%where
%\begin{enumerate}
%	\item[$\bullet$] $\ds a_{i_0}^{(n)}$ is the point added to the greedy sequence at the $n$-th iteration, in other words, it is the point $a_{n}$, 
%	\item [$\bullet$] $\ds a_{i_0-1}^{(n)}$ and $a_{i_0+1}^{(n)}$ are the points respectively lower and greater than $ a_{i_0}^{(n)}$, \\  i.e. $ a_{i_0-1}^{(n)}<a_{i_0}^{(n)}<  a_{i_0+1}^{(n)} $,
%	\item[$\bullet$] \begin{equation}
%	\label{poidsinfetsup}
%	p_{-}^{n} = \mu\left(\left[a_{i_0-\frac{1}{2}}^{(n)}, a_{mil}^{(n)}\right]\right) \qquad \mbox{and} \qquad p_{+}^{n} = \mu\left(\left[a_{mil}^{(n)},a_{i_0+\frac{1}{2}}^{(n)}\right]\right).
%	\end{equation}
%	where 
%	$\ds a_{i_0 \pm \frac{1}{2}}^{(n)}=\frac{a_{i_0}^{(n)}+a_{i_0 \pm 1}^{(n)}}{2}$  and  $\ds a_{mil}^{(n)}=\frac{a_{i_0 + 1}^{(n)}+a_{i_0 - 1}^{(n)}}{2}$, noting always $a_0=-\infty$ and $a_{n}=+\infty$. \\ 
%\end{enumerate}
\begin{eqnarray}
\label{caliteratifd1}
\ds I_{n}(f) &=& I_{n-1}(f) - p_{-}^{n}f(a_{i_0-1}^{(n)})-p_{+}^{n}f(a_{i_0+1}^{(n)})+(p_{+}^{n}+p_{-}^{n})f(a_{i_0}^{(n)}) \nonumber \\
&=& I_{n-1}(f)- p_{-}^{n}(f(a_{i_0-1}^{(n)})-f(a_{i_0}^{(n)}))-p_{+}^{n}(f(a_{i_0+1}^{(n)})-f(a_{i_0}^{(n)})),
\end{eqnarray}
where
\begin{enumerate}
	\item[$\bullet$] $\ds a_{i_0}^{(n)}$ is the point added to the greedy sequence at the $n$-th iteration, in other words, it is the point $a_{n}$, 
	\item [$\bullet$] $\ds a_{i_0-1}^{(n)}$ and $a_{i_0+1}^{(n)}$ are the points  lower and greater than $ a_{i_0}^{(n)}$, i.e. $ a_{i_0-1}^{(n)}<a_{i_0}^{(n)}<  a_{i_0+1}^{(n)} $,
	\item[$\bullet$] \begin{equation}
	\label{poidsinfetsup}
	p_{-}^{n} = P\big(\big[a_{i_0-\frac{1}{2}}^{(n)}, a_{\text{mil}}^{(n)}\big]\big) \qquad \mbox{and} \qquad p_{+}^{n} = P\big(\big[a_{\text{mil}}^{(n)},a_{i_0+\frac{1}{2}}^{(n)}\big]\big).
	\end{equation}
	where 
	$\ds a_{i_0 \pm \frac{1}{2}}^{(n)}=\frac{a_{i_0}^{(n)}+a_{i_0 \pm 1}^{(n)}}{2}$  and  $\ds a_{\text{mil}}^{(n)}=\frac{a_{i_0 + 1}^{(n)}+a_{i_0 - 1}^{(n)}}{2}$, with $a_0=-\infty$ and $a_{n}=+\infty$.
\end{enumerate}
Practically, this numerical iterative method can be applied without storing the whole ordered greedy quantization sequence nor computing the weights of the Vorono\"i cells, which could appear as significant drawbacks for quantization. Instead, it requires the possession of $2$ indices of $2$ particular points of the non-ordered greedy quantization sequence and $2$ weights. In fact, one can start by determining the indices $\overline{n}$ and $\underline{n}$ of the points preceding and following $a_{n}$ in the ordered sequence, in other words, the points in the non-ordered sequence corresponding to $a_{i_0-1}^{(n)}$ and $a_{i_0+1}^{(n)}$. Then, we will be able to compute the weights $p_{-}^{n}$ et $p_{+}^{n}$ to finally proceed with the iterative approximation of $I(f)$ according to~$(\ref{caliteratifd1})$.
\subsection{Product greedy quantization ($d>1$)}
\label{dimensiond}
In higher dimensions, greedy quantization has always the recursive properties, so it gets interesting to apply the same numerical improvements as in the one-dimensional case. However, the construction of multidimensional greedy quantization sequences is complex and expensive since it relies on complicated stochastic optimization algorithms. As an alternative, one can use one-dimensional greedy quantization grids as tools to obtain multidimensional greedy quantization sequences in some cases.
\subsubsection{How to build multi-dimensional greedy product grids}
Multidimensional greedy quantization sequences can be obtained as a result of the tensor product of one-dimensional sequences, when the target law is a tensor product of its independent marginal laws. These grids are, of course, not optimal nor asymptotically optimal but they allow to approach the multidimensional law.\\
Let $X_1,\ldots, X_d$ be  $d$ independent $L^2$-random variables taking values in $\R$ with respective distributions $\mu_1,\ldots, \mu_d$ and $a^{1,(n_1)}, \ldots, a^{d,(n_d)}$ the corresponding greedy quantization sequences. By computing the tensor product of the $d$ one-dimensional greedy sequences of the laws $\mu_1,\ldots,\mu_d$, we obtain the $d$-dimensional greedy quantization grid $a^{1,(n_1)} \otimes \ldots \otimes a^{d,(n_d)}$ of the product law $\mu=\mu_1 \otimes \ldots \otimes \mu_d$, given by $\big(a_{\underline{j}}^{(n)}\big)_{1 \leq \underline{j}\leq n}=\big(a_{j_1}^{1,(n_1)}, \ldots, a_{j_d}^{d,(n_d)}\big)_{1 \leq j_1 \leq n_1, \ldots, 1\leq j_d \leq n_d}$ of size $\ds n=\prod_{i=1}^{d}n_i$. The corresponding quantization error is given by 
\begin{equation}
\label{proderror}
e_r(a^{1,(n_1)} \otimes \ldots \otimes a^{d,(n_d)}, X_1 \otimes \ldots \otimes X_d)^r = \sum_{k=1}^{d}e_r(a^{k,(n_k)},X_k)^r.
\end{equation}
Moreover, the weights $p_{\underline{j}}^{(n)}$ of the $d$-dimensional Vorono\"i cells $\big(W_{\underline{j}}\big(a^{(n)}\big)\big)_{1 \leq \underline{j} \leq n}$ can be computed from the weights  $p^{k,n_k} ,\, k=1,\ldots,d$, of the Vorono\"i cells $\big(W^{k,n_k}_{i}\big(a^{k,(n_k)}\big)\big)_{1 \leq i \leq n_k}$ of each one-dimensional greedy quantization sequence, via
$$\ds p_{\underline{j}}=p^{1,n_1}_{j_1} \times \ldots \times p_{j_d}^{d,n_d} \qquad \forall j_k \in \{1, \ldots,n_k\}, \forall k \in \{1,\ldots,d\}, \, \forall \underline{j} \in \{1,\ldots,n\}.$$
The implementation of $d$-dimensional grids is not a point-by-point implementation. In fact, at each iteration $n$, $a^{(n)}$ is obtained from $a^{1,(n_1)},\ldots, a^{d,(n_d)}$, keeping in mind that $n_1 \times \ldots \times n_d =n$. Having the $d$ one-dimensional sequences, one must add a point to one one-dimensional sequence, generating this way several points of the multidimensional sequence. At this step, we must choose between $d$ possibilities: adding one point to only one sequence $a^{k,(n_k)}$ among the $d$ marginal sequences, obtaining $\ds a^{(n_1\times\ldots\times n_{k-1}\times(n_k+1)\times n_{k+1} \times \ldots  \times n_d)}$. These $d$ cases are not similar since each one produces a different error quantization. So, the implementation is not a random procedure. To make the right decision, one must compute in each case, using~$(\ref{proderror})$, the quantization error $E_k$ obtained if we add a point to $a^{k,(n_k)}$ for a $k \in \{1,\ldots,d\}$. In other words, we compute, for $k=1,\ldots,d$
$$
E_k  = e_r(a^{k,(n_k+1)},\mu_k)^r+\sum_{l\in\{1,\ldots, d\}\setminus \{k\}}e_r(a^{l,(n_l)},\mu_l)^r.
$$
Then, one choses the index $i$ such that $\ds E_{i}=\min_{1\leq k \leq d}E_k$ and, so, one adds a point to the sequence $a^{i,(n_{i})}$ and obtains the grid $\ds a^{(n_1\times\ldots\times n_{i-1}\times(n_{i}+1)\times n_{i+1} \times \ldots  \times n_d)}$.\\
We note that if the marginal laws $\mu_1,\ldots ,\mu_d$ are identical, this step is not necessary and the choice of the sequence to which a point is added, at each iteration, is systematically done in a periodic manner.
\subsubsection{Numerical integration}
Similarly to the $1$-dimensional case, the majority of the Vorono\"i cells do not change while passing from an iteration $n$ to an iteration $n+1$. %With the creation of a few new cells, only neighboring cells are modified. In fact,
At the $n$-th iteration, having $n_1 \times \ldots \times n_d$ points in the sequence, one adds a new point to $a^{(i,n_i)}$. Hence, we will have $n_1\times\ldots\times n_{i-1}\times n_{i+1}\times \ldots \times n_d$ new created cells having for centroids the new points added to the $d$-dimensional sequence $a^{(n)}$, and another $2(n_1\times\ldots\times n_{i-1}\times n_{i+1}\times \ldots \times n_d)$ modified cells, corresponding to all the neighboring cells of the new added cells. In total, there is $3(n_1\times\ldots\times n_{i-1}\times n_{i+1} \times \ldots \times n_d)$ new Vorono\"i cells, while the rest of the cells remain unchanged. 
This leads to an iterative formula for quantization-based numerical integration (where the same principle as in the one dimensional case is applied) % At the iteration $n$ to $n+1$, we start by determining the indices $\underline{n_i+1}$ et $\overline{n_i+1}$, in the non-ordered one-dimensional quantization sequence, of the elements preceding and following the new point $a_{n_i+1}^{i,(n_i+1)}$ added to the $1$-dimensional sequence at this iteration, i.e. the indices corresponding to $a_{i_0+1}^{i,(n_i+1)}$ et $a_{i_0-1}^{i,(n_i+1)}$, where $i_0$ corresponds to the index, in the ordered one-dimensional sequence, of the added point, as well as the weights $p_{-}^{i,n_i+1}$ et $ p_{+}^{i,n_i+1}$ via the same equations used when $d=1$.
as follows
\begin{align}
\label{calculiteratifd}
\ds I_{n+1}(f) & =I_n(f)-p_{-}^{i,n_i+1}\left(\sum_{\substack{j_k=1 \\ k\in \{1,\ldots,d\} \setminus\{i\}}}^{n_k} \prod_{\substack{k=1,\ldots,d \\ k\neq i}}p_{j_k}^{k,(n_k)} \left( f(a_{j_1}^{1,(n_1)},\ldots,a_{i_0-1}^{i,(n_i+1)},\ldots,a_{j_d}^{d,(n_d)}) \right. \right. \nonumber \\
& \left. \left. -f(a_{j_1}^{1,(n_1)},\ldots,a_{i_0}^{i,(n_i+1)},\ldots, a_{j_d}^{d,(n_d)}) \right) \vphantom{(\sum_{\substack{j_k=1 \\ k\in \{1,\ldots,d\} \setminus\{i\}}}^{n_k}} \right) \nonumber \\
& -p_{+}^{i,n_i+1}\left( \sum_{\substack{j_k=1 \\ k\in \{1,\ldots,d\} \setminus\{i\}}}^{n_k } \prod_{\substack{k=1,\ldots,d \\ k\neq i}}p_{j_k}^{k,(n_k)} \left( f(a_{j_1}^{1,(n_1)},\ldots,a_{i_0+1}^{i,(n_i+1)},\ldots,a_{j_d}^{d,(n_d)}) \right. \right.  \nonumber \\
& \left. \left. -f(a_{j_1}^{1,(n_1)},\ldots,a_{i_0}^{i,(n_i+1)},\ldots, a_{j_d}^{d,(n_d)}) \right) \vphantom{(\sum_{\substack{j_k=1 \\ k\in \{1,\ldots,d\} \setminus\{i\}}}^{n_k}}\right) 
\end{align}
Note that in the $d$-dimensional case, the use of the weights $p^{k,(n_k)}$ for $k \in \{1,\ldots,d\} \setminus \{i\}$ of the Vorono\"i cells of the other marginal sequences obtained at the previous iteration is essential, as well as the use the ordered one-dimensional greedy sequences $a^{k,(n_k)}$ for $k \in \{1,\ldots,d\} \setminus \{i\}$.
\section{Numerical applications and examples}
\label{numericalexample}
\subsection{Greedy quantization sequences for Gaussian distribution via Box-M\"uller}
\label{boxmuller}
The Box-M\"uller method allows to generate a random vector with normal distribution $\mathcal{N}(0,I_2)$, actually two independent one-dimensional random variables $Z_1$ and $Z_2$ with distribution $\mathcal{N}(0,1)$ by considering two independent random variables $E$ and $U$ with respective distributions $\mathcal{E}(1)$ and $\mathcal{U}([0,1])$. Then, $2E \sim \mathcal{E}(\frac{1}{2})$ and $2 \pi U \sim \mathcal{U}([0,2\pi])$, so, the two variables
$$Z_1=\sqrt{2E}\, \cos (2\pi U)\qquad \mbox{et} \qquad Z_2=\sqrt{2E}\, \sin(2 \pi U)$$
are independent and with normal distribution $\mathcal{N}(0,1)$. \\
%Note that we can use two variables $U_1$ et $U_2$ i.i.d. with uniform distribution $\mathcal{U}([0,1])$. In fact, if $E=-2\ln(U_1)$ and $U=2\pi U_2$, then $E\sim\mathcal{E}(\frac{1}{2})$ and $U\sim \mathcal{U}([0,2\pi])$ and the rest will be the same.\\
In order to apply greedy properties, we use greedy quantization sequences $\varepsilon^{(n_1)}$ and $u^{(n_2)}$ of respective distributions $\mathcal{E}(1)$ and $\mathcal{U}[0,1]$ to design two $\mathcal{N}(0,1)$-distributed independent sequences $z_1^{(n)}$ et $z_2^{(n)}$, of size $n=n_1 \times n_2$, via the previous formulas so we can get a greedy sequence $z^{(n)}$ of the two-dimensional normal distribution $\mathcal{N}(0,I_2)$. 
The procedure is implemented as described in section~$\ref{dimensiond}$. At each iteration, we must choose the one-dimensional distribution to which we should add a point. Thus, we compute the error induced if we add a point to $u^{(n_2)}$
$$E_u = e_2\left(u^{(n_2+1)},\mathcal{U}[0,2\pi]\right)^2+e_2\left(\varepsilon^{(n_1)},\mathcal{E}\left(\frac{1}{2}\right)\right)^2
= 4\pi^2 e_2\left(u^{(n_2+1)},\mathcal{U}[0,1]\right)^2+4e_2\left(\varepsilon^{(n_1)},\mathcal{E}\left(1\right)\right)^2,$$
and the error induces if we add a point to $\varepsilon^{(n_1)}$
$$E_{\varepsilon}  =e_2\left(u^{(n_2)},\mathcal{U}[0,2\pi]\right)^2+e_2\left(\varepsilon^{(n_1+1)},\mathcal{E}\left(\frac{1}{2}\right)\right)^2=  4\pi^2e_2\left(u^{(n_2)},\mathcal{U}[0,1]\right)^2+4e_2\left(\varepsilon^{(n_1+1)},\mathcal{E}\left(1\right)\right)^2,$$
and we add a point to $u^{(n_2)}$ if $E_u<E_{\varepsilon}$ and a point to $\varepsilon^{(n_1)}$ if $E_{\varepsilon}<E_u$. \\
 
To design sequences in dimension $d >2$, one uses several couples $(E_i,U_i)$ to get several pairs $(Z_i,Z_j)$ and use the wanted number of $(Z_k)_{k}$ to obtain multidimensional sequences. %For example, to design a sequence for $\mathcal{N}(0,I_5)$, we use $3$ exponential one-dimensional sequences and $3$ uniform one-dimensional sequences, we compute $6$ one-dimensional normal sequences and, finally, use $5$ of them to perform the computations.\\
In figure~$\ref{BM3}$, we compare two greedy quantization sequences of the distribution $\mathcal{N}(0,I_3)$ of size $N=15^3$, one is obtained using the Box-M\"uller method based on two greedy exponential sequences $\mathcal{E}(1)$ and two greedy uniform sequences $\mathcal{U}([0,1])$, and the other obtained by greedy product quantization based on $3$ one-dimensional Gaussian greedy sequences. The weights of the Vorono\"i cells in both cases are represented by a color scale (growing from blue to red) and we observe that the weights of the cells in the center have the highest values and those values decrease as long as we sweep away to the borders, as expected for a Gaussian distribution.\\
We should also note that, even if the greedy product quantization of a Normal distribution takes the shape of a cube (which is unusual fo such distribution), the low values of the Vorono\"i weights at the edges of this cube allow to consider such a sequence as a valid approximation of the Gaussian distribution. 
\begin{figure}
	\begin{center}
			%\vspace{-1cm}
		\begin{tabular}{cc}
			\includegraphics[width=.4\textwidth,height=.25\textheight,angle=0]{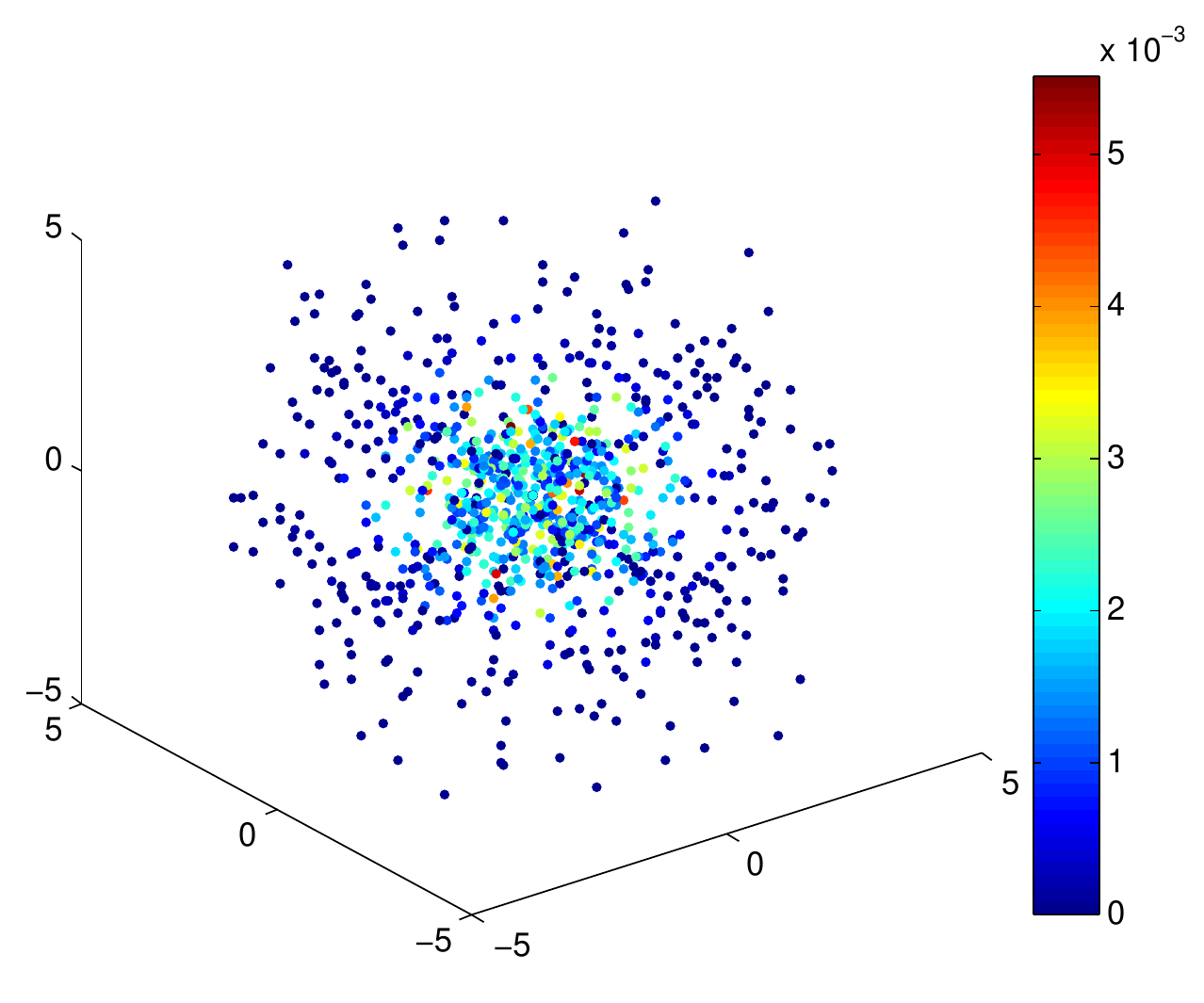}&
			\hspace{1cm}
			\includegraphics[width=.4\textwidth,height=.25\textheight,angle=0]{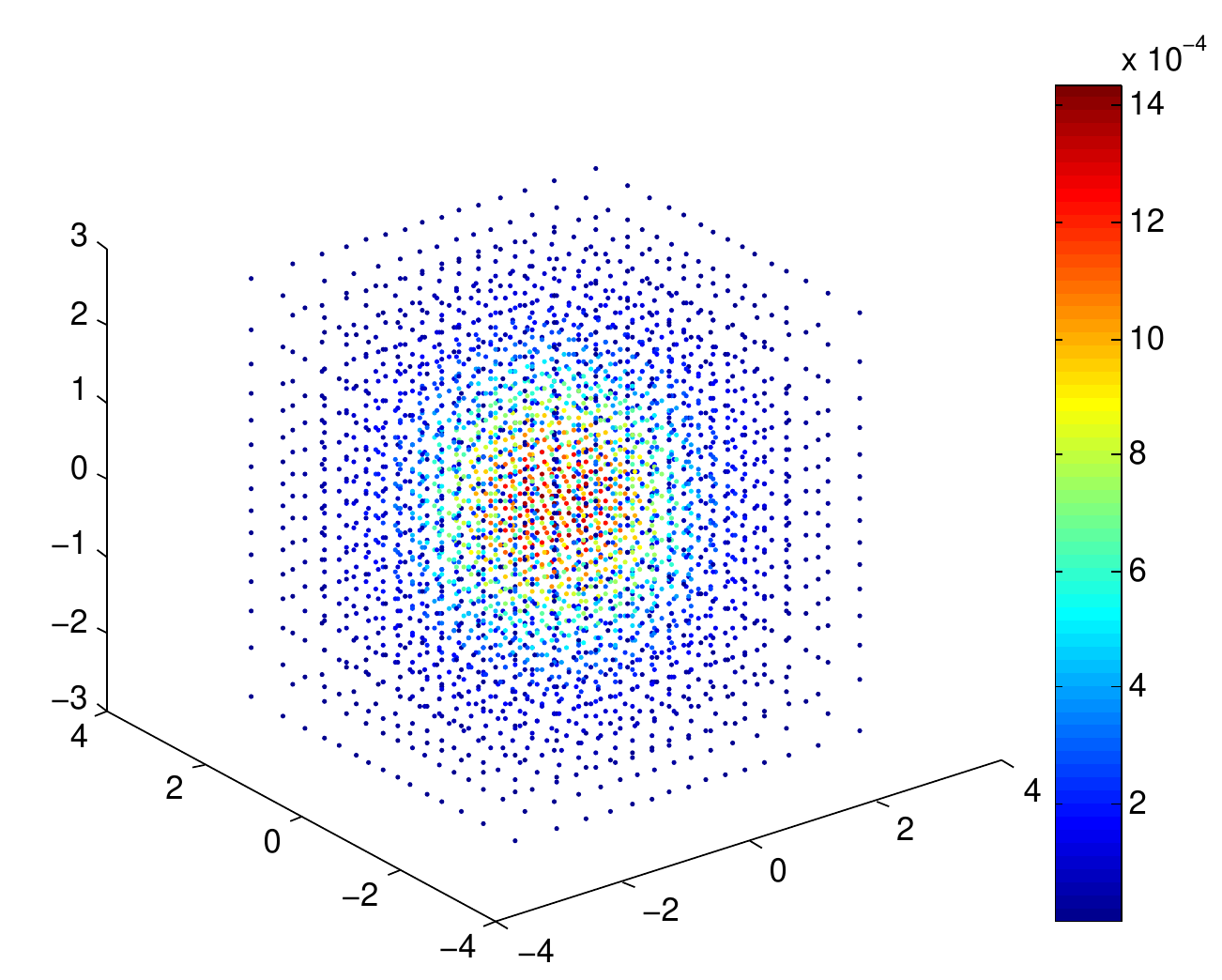}\\
		\end{tabular}
	\end{center} 
	%\vspace{-1cm}
	\caption{Greedy quantization sequences of the distribution $\mathcal{N}(0,I_3)$ of size $N=15^3$ designed by Box M\"uller method (left) and greedy product quantization (right).}
	\label{BM3}
\end{figure}

\subsection{Pricing of a $3$-dimensional basket of European call options }
We consider a Call option on a basket of $3$ positive risky assets, with strike price $K$ and maturity $T$, with payoff
$h_T=\left( \sum_{i=1}^{3} w_i X_i -K \right)_+$
where $(X_1,X_2 ,X_3)$ represent the prices of the $3$ traded assets of the market and $w_i$ are positive weights such that $ \sum_{i=1}^{3}w_i=1$.
We consider a $3$-dimensional correlated Black-Scholes model where the prices of the assets are given by
$$dX_i=X_{i,t} \left(rdt +\sigma_idW_{i,t}\right), \qquad X_{i,0}=x_{i,0}.$$
where $r$ is the interest rate, $\sigma_i$ the volatility of $X_i$ and the $(W_i)_i$ represent a correlated $3$-dimensional Brownian motion, i.e. $(W_i,W_j)=\rho_{ij}t.$
Then, one has for every $i\in \{1,2,3\}$
$$\ds X_i=X_{0,i} e^{(r-\frac{\sigma_i^2}{2})t+\sum_{j=1}^{q}\sigma_{ij} W_{j,t}}, \quad t \in [0,T].$$
\begin{figure}
	\begin{center}
		\includegraphics[width=.7\textwidth,height=.3\textheight,angle=0]{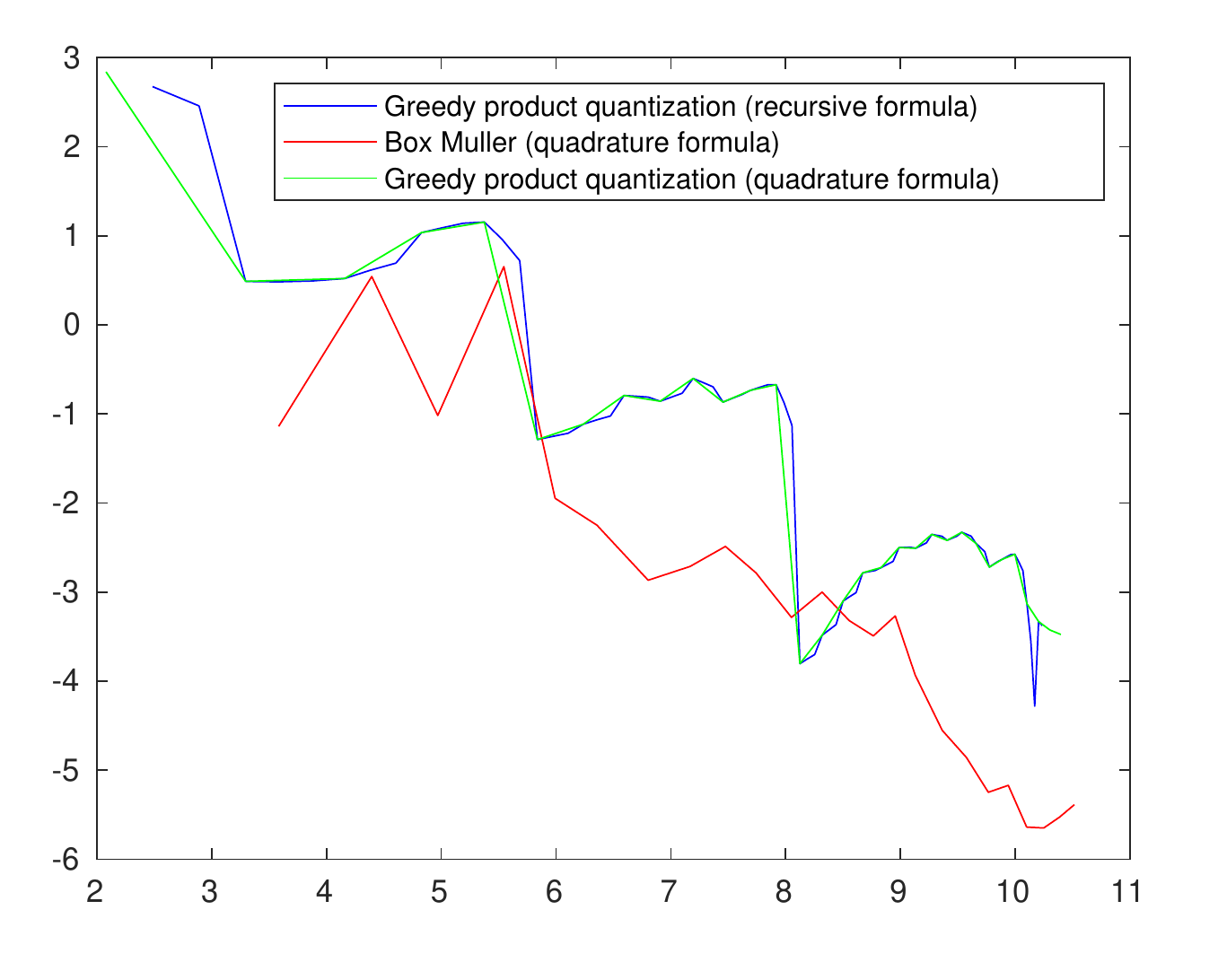}
	\end{center} 
	\vspace{-1cm}
	\caption{Errors induced by the pricing of a $3$-dimensional basket of call options $V_0$ in a Black-Scholes model computed using $3$-dimensional greedy normal quantization sequences obtained by Box-M\"uller via quadrature formula and by greedy product quantization via a recursive formula and via a quadrature formula for $n=1, \ldots, 30\, 000$ (logarithmic scale).}
	\label{basketdim3}
\end{figure}
Our aim is to compute $$V_0=e^{-rT} \E[h_T(X_T)]=e^{-rT} \E[h_T(X_1,X_2,X_3)]$$ relying on the greedy quantization sequences and using the recursive formula for numerical integration introduced in the previous sections. \\
First, we estimate $V_0$ by a quadrature formula according to~$(\ref{quadratureformula})$
\begin{equation*}
\ds V_0 =  \sum_{j_1=1}^{n} \sum_{j_2=1}^{n} \sum_{j_3=1}^{n} \phi(z_{j_1}^{1,(n)},z_{j_2}^{2,(n)},z_{j_3}^{3,(n)}) p_{j_1}^{1,(n)} p_{j_2}^{2,(n)} p_{j_3}^{3,(n)}.
\end{equation*}
where $(z^{1,(n)},z^{2,(n)},z^{3,(n)})$ is a $3$-dimensional greedy quantization sequences of the Gaussian distribution $\mathcal{N}(0,I_3)$ obtained, on one hand, by the Box-M\"uller algorithm relying on $2$ one-dimensional exponential greedy sequence $\varepsilon^{(n_1)}, \, n_1=16$ and $2$ one-dimensional uniform greedy sequence $u^{(n_2)}, \, n_2=15$ and, on the other hand, by greedy product quantization of $3$ one-dimensional sequences of size $32$ each. Then, we estimate $V_0$ by the recursive formula~$(\ref{calculiteratifd})$ for $d=3$ using the greedy product quantization sequence. 
%For a start, we need to write the payoff as a function $\phi$ of $3$ independant random variables $(Z_1,Z_2,Z_3)$ with standard normal distribution. Then, we use $3$-dimensional greedy quantization sequences $z=(z^{1,(n)},z^{2,(n)},z^{3,(n)})$ of the gaussian distribution $\mathcal{N}(0,I_3)$ obtained by greedy product quantization on the one hand, and by the Box-Muller algorithm on the other hand, to approximate 
%\begin{equation}
%\label{V0}
%V_0 =e^{-rT}\E[h_T(X_T)]  =\E[\phi(Z_1,Z_2,Z_3)].
%\end{equation}
%We estimate the expectation~$(\ref{V0})$ first by a quadrature formula according to~$(\ref{quadratureformula})$
%\begin{equation*}
%\ds V_0 =  \sum_{j_1=1}^{n} \sum_{j_2=1}^{n} \sum_{j_3=1}^{n} \phi(z_{j_1}^{1,(n)},z_{j_2}^{2,(n)},z_{j_3}^{3,(n)}) p_{j_1}^{1,(n)} p_{j_2}^{2,(n)} p_{j_3}^{3,(n)}.
%\end{equation*}
%where $(z^{1,(n)},z^{2,(n)},z^{3,(n)})$ is obtained by the Box-Muller algorithm (explained in section~$\ref{boxmuller}$), relying on $2$ one-dimensional exponential greedy sequence $\varepsilon^{(n_1)}, \, n_1=16$ and $2$ one-dimensional uniform greedy sequence $u^{(n_2)}, \, n_2=15$. 
%Moreover, a second approximation via this quadrature formula is implemented where $z^{(n)}=(z^{1,(n)},z^{2,(n)},z^{3,(n)})$ is the $3$-dimensional greedy product quantization sequence of $\mathcal{N}(0,I_3)$.
%
%Also, we estimate $V_0$ by the recursive formula~$(\ref{calculiteratifd})$ for $d=3$ for the sequence obtained by greedy product quantization implemented using $3$ one-dimensional normal greedy sequences of size $32$ each. \\
We obtain sequences of size $32\, 000$ and we consider 
$$r=0.1 \; , \; \sigma_i=0.3 \; ,\; X_{i,0}=100 \; ,\; T=1 \; \mbox{and} \; K=100.$$
Moreover, we consider a Brownian motion such that $\rho_{1,1}=\rho_{1,2}=\rho_{1,3}=0.5$ and all the others $\rho_{i,j}$'s are equal to $0$. The reference price is given by a large Monte Carlo simulation with control variate of size $M=2.10^7$. We consider the control variate 
$$k_T=\left(e^{\sum_{i=1}^3 w_i\log(X_i)}-K\right)_+$$
which is positive and lower than $h_T$ owing to the convexity of the exponential. Since $e^{-rT}\E k_t$ has a normal distribution with mean $(r-\frac12\sum_{i=1}^3 w_i \sigma_i^2)T$ and variance $w^t \sigma \sigma^t w T$, it admits a closed form given by 
\begin{equation*}
e^{-rT}\E k_t=\mbox{Call}_{\mbox{BS}}\left(\prod_{i=1}^3 X_{i,0}^{w_i} e^{-\frac12 T (\sum_{i=1}^3 w_i \sigma_i^2-w^t \sigma \sigma^t w)}, K,r,\sqrt{w^t \sigma \sigma^t w},T \right).
\end{equation*}
\begin{table}
	\begin{center}
		\begin{tabular}{|l|l|l|l|}
			\hline
			$n$         &  BM    &    GPQ     &     GPI  \\      
			\hline
			$100$   & $1.72$  & $1.68$   &  $1.84$          \\
			\hline
			$1\, 000$     & $0.07$  & $0.42$ & $0.42$       \\
			\hline
			$8\, 000$  &  $0.04$  &  $0.08$  &  $0.08$       \\
			\hline 
			$15\, 000$ & $0.07$ & $0.08$ & $0.08$\\
			\hline
		\end{tabular}
	\end{center}
	\caption{Errors of the approximation of a $3$-dimensional basket of call options $V_0$ in a Black-Scholes model by Box-M\"uller with quadrature formula (BM), greedy product quantization with quadrature formula (GPQ) and greedy product quantization with recursive formula (GPI) for different number of points $n$.} 
	\label{tabbasketdim3}
\end{table}
We compare the three methods in figure~$\ref{basketdim3}$ where we represent, in a logarithmic scale, the error induced by each method as a function of the number of points varying between $1$ and $32\,000$ and in table~\ref{tabbasketdim3} where we expose the errors obtained by each method for some particular number of points.
The recursive numerical integration gives the same results as the quadrature formula-based numerical integration making quantization-based numerical integration less expensive and more advantageous by reducing the cost in time and storage. Moreover, one deduces that the Box-M\"uller algorithm is more accurate than the greedy product quantization. This can be explained by the fact that Box-M\"uller sequences fill the space in a way that resembles more to the normal distribution, we can notice a kind of ball different than the cube observed when implementing greedy product sequences (see figure~\ref{BM3}).
\section{Further properties and numerical remarks}
\label{numericalproperties}
In this section, we present, based on numerical experiments, some properties of the one-dimensional quadratic greedy quantization sequences. 
%We explore, based on numerical experiments, the existence of sub-optimal sequences of the greedy quantization grids, as well as the convergence of empirical measures. Moreover, we expose some numerical tests to study the stationarity of these sequences or to see if they behave asymptotically like stationary quantizers. Finally, a comparison with Quasi-Monte Carlo methods is established. 
We recall that $a^{(n)}=\{a_1^{(n)},\ldots,a_n^{(n)}\}$ denotes the reordered greedy sequence of the $n$ first elements $\{a_1,\ldots,a_n\}$ of $(a_n)_{n \geq 1}$.
\subsection{Sub-optimality of greedy quantization sequences}
The implementation of a greedy quantization sequence $(a_n)_{n\geq 1}$ of a distribution $P$ and the computation of the corresponding weights $p_i^{n}$ of the Vorono\"i cells $W_i(a^{(n)})$ for $i \in \{1,\ldots, n\}$ defined by~$(\ref{Voronoicells})$ is, in general, not optimal. However, numerical implementations and graphs representing 
$a_i \mapsto p_i^n=P(X \in W_i(a^{(n)}))$
for different number of points $n$ show that, for certain distributions, the weights of the Vorono\"i cells converge towards the density curve of the corresponding distribution when the greedy sequence has a certain number of points.\\
For the normal distribution, 
%numerical experiments show that the curve of the weights takes the form of the bell curve 
this is observed when the size of the sequence is equal to $n=2^k-1$, for every integer $k \geq 1$. So, we can say that the greedy quantization sequence is sub-optimal since the subsequence 
\begin{equation}
\label{suiteoptnorm}
\ds \alpha^{(n)}= \alpha^{(2^k-1)} \, \mbox{t.q.} \quad  n=2^k-1,\, k \in \N^*
\end{equation}
is itself optimal. %  and converges well to the normal distribution.
% In fact, these values of $n$ correspond to the minimas observed on the graph representing $n\rightarrow e_r(a^{(n)},P)$.\\
%Regarding the uniform distribution on $[0,1]$, we observe similar results. %Numerical experiments show that the $L^2$-quantization error admits minimas for certain numbers of points $n$ and that, for the same numbers $n$, the weights of the Vorono\"i cells converges towards the density curve of the uniform distribution. In this case, 
Regarding the uniform distribution on $[0,1]$, we can check that there exists $2$ sub-optimal sequences of the greedy sequence defined by
\begin{multicols}{2}\noindent
%\begin{equation}
\[
\label{suiteopt1unif}
\left\{
\begin{array}{lr}
\alpha_0=3, \\
\alpha_n=2\alpha_{n-1}+1  & \mbox{if } n \equiv 1 \pmod 3,\\
\alpha_n=2(\alpha_{n-1}-2)+1 &  \mbox{if }n \equiv 2 \pmod 3,\\
\alpha_n=2(\alpha_{n-1}+2)+1  & \mbox{if } n \equiv 0 \pmod 3.
\end{array}
\right.
\]%\end{equation}
%and
\[%\begin{equation}
\label{suiteopt2unif}
\left\{
\begin{array}{lr}
\alpha_0=11, \\
\alpha_n=2\alpha_{n-1}+1  & \mbox{if }n \equiv 1 \pmod 3,\\
\alpha_n=2(\alpha_{n-1}-2)+1  & \mbox{if } n \equiv 2 \pmod 3,\\
\alpha_n=2(\alpha_{n-1}+2)+1  & \mbox{if }n \equiv 0 \pmod 3.
\end{array}
\right.
\]%\end{equation}
\end{multicols}
These results explain, in a certain way, the cycloid aspect of the graphs of the quantization error, the points at which the quantization error reaches its minimum correspond to the optimal subsequences defined above.%by~$(\ref{suiteoptnorm})$,~$(\ref{suiteopt1unif})$ et~$(\ref{suiteopt2unif})$.\\

Some results for the normal distribution are represented in figure~$\ref{weightssuboptimal}$ where we observe the unimodal weights for $n=255=2^8-1$ %and $n=1023=2^{10}-1$ 
and non-unimodal weights for $n=400$. %and $n=2100$.
\begin{figure}
	\begin{center}
		\vspace{-2cm}
		\begin{tabular}{cc}
			\includegraphics[width=.4\textwidth,height=.4\textheight,angle=0]{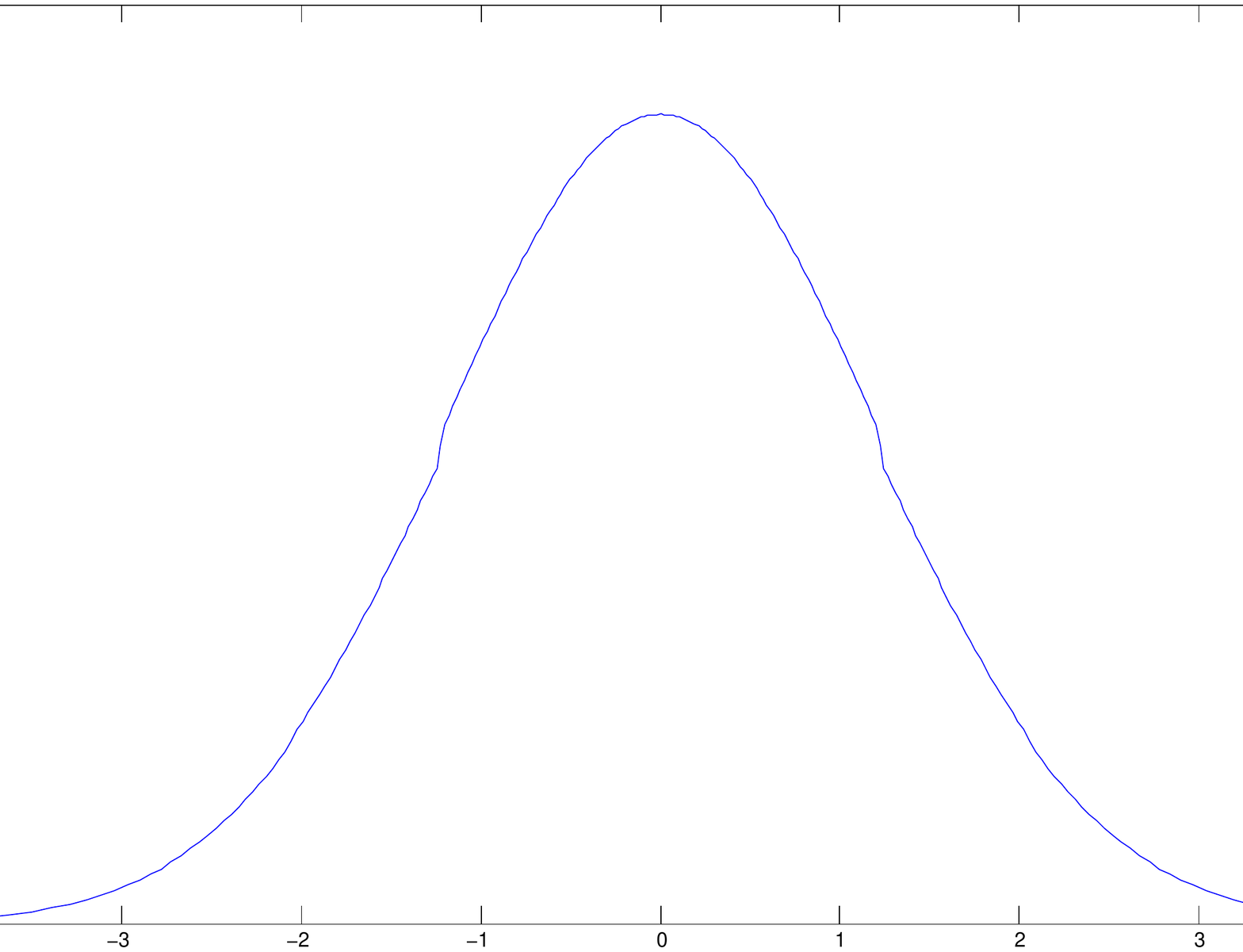}&
			\hspace{0.5cm}
			\includegraphics[width=.4\textwidth,height=.4\textheight,angle=0]{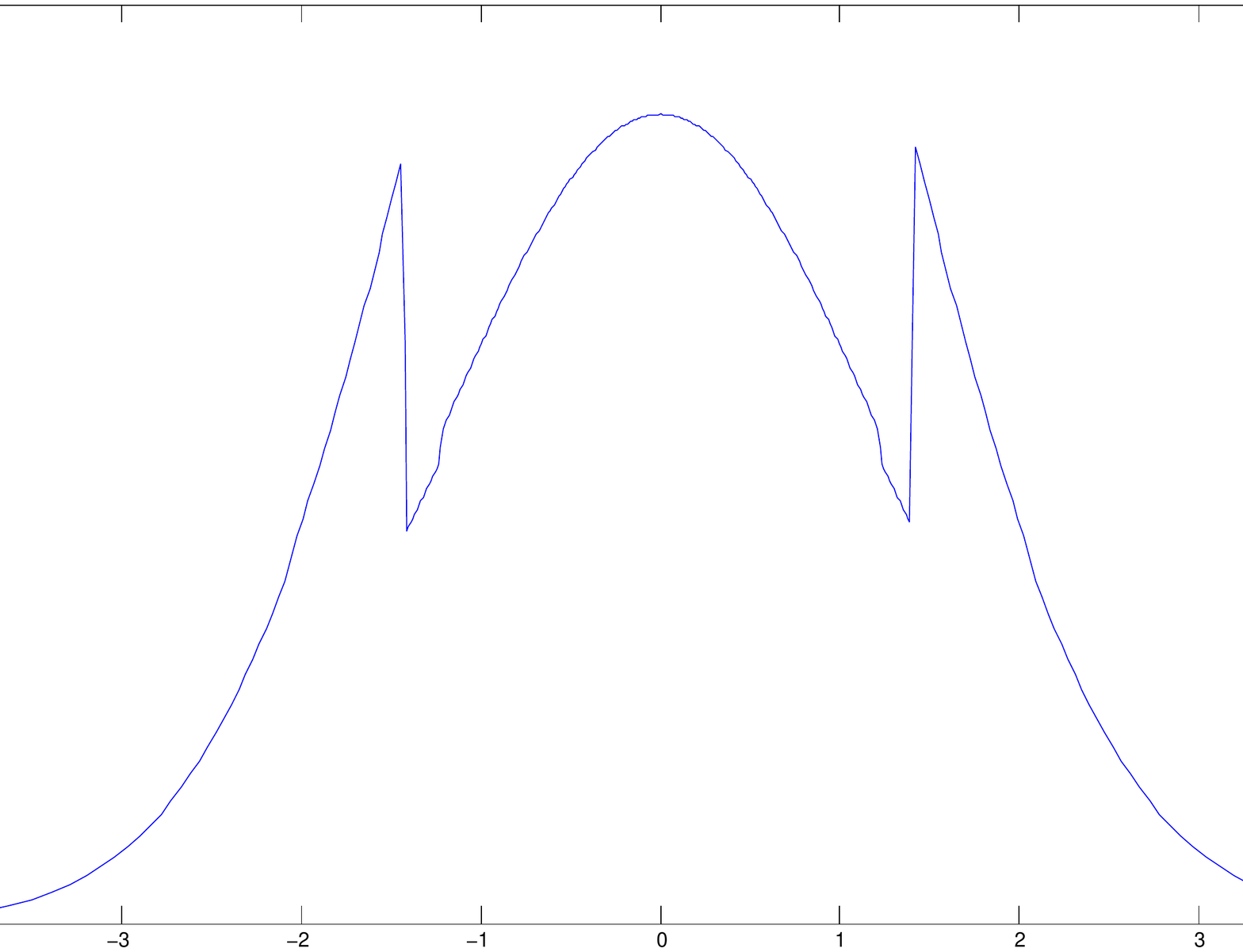}\\

		\end{tabular}
	\end{center} 
	\vspace{-2.5cm}
	\caption{Representation of $a_i \mapsto p_i^n$ where $(p_i^n)_{1 \leq i \leq n}$ denote the Vorono\"i weights corresponding to the greedy quantization sequence of the normal distribution $\mathcal{N}(0,1)$ implemented by Lloyd's algorithm for $n=255$ (left), $n=400$ (right).}
	%, $n=1023$ (down, left) et $n=2100$ (down, right).}
	\label{weightssuboptimal}
\end{figure}

Similarly, the greedy quantization sequence of the Laplace distribution with parameters $0$ and $1$ admits optimal subsequences taking the form $a^{(2^k-1)}, k \in \N^*$. These observations allow to conjecture the sub-optimality of such subsequences for symmetrical distributions around $0$.
\subsection{Convergence of standard and weighted empirical measures}
The existence of suboptimal greedy quantization sequences, as detailed previously, gives a motivation to dig deeper and study the empirical measures associated to a greedy quantization sequences. % We start by recalling a few results for optimal quantization.\\
In fact, sequences of asymptotically optimal $n$-quantizers $(\Gamma_n)_{n \geq 1}$, of an absolutely continuous distribution $P$ w.r.t. the Lebesgue measure with density $f$, satisfy some empirical measure convergence theorems established in \cite{GraLu00} (see theorem $7.5$ p. $96$) and \cite{Delattre04} and recalled below, where
%Let $X$ be a r.v. with distribution $P$ absolutely continuous w.r.t. the Lebesgue measure with density $f$. A sequence $(\Gamma_n)_{n \geq 1}$ of quantizers is called an asymptotically optimal $n$-quantizer for $P$ if $ \E \min_{x_i \in \Gamma_n} V(\|X-x_i\|)< +\infty$ for every $n$ and $\E \min_{x_i \in \Gamma_n} V(\|X-x_i\|) \sim e_{n,V}(P)$ as $n \rightarrow +\infty$ where $V:\R_+\rightarrow \R_+$ is a nondecreasing function such that $V(0)=0$ and $V(t)>0$ for $t>0$. Such sequences satisfy some empirical measure convergence theorems established in \cite{GraLu00} (see theorem $7.5$ p. $96$) and \cite{Delattre04} recalled below where
$$\ds \widehat{P}_n=\frac{1}{n} \sum_{i=1}^{n} \delta_{x_i^{n}} \qquad \mbox{ and } \qquad  \widetilde{P}_n=\sum_{i=1}^{n} P(W_i(\Gamma_n)) \delta_{x_i^{n}}$$ 
designate, respectively, the empirical standard measure and the empirical weighted measure associated to $\Gamma_n=\{x_1^n,\ldots,x_n^n\}$.%, where $(W_i(\Gamma_n))_{1 \leq i \leq n}$ is the corresponding Vorono\"i diagram and we recall the results in the following theorems.
\begin{thm}
	Assume $P$ is absolutely continuous w.r.t the Lebesgue measure on $\R^d$ with density $f$. Let $\Gamma_n$ be an asymptotically optimal $n$-quantizer of $P$. Then, denoting $C=\big(\int_{\R} f^{\frac{d}{d+p}}(u) du \big) ^{-1}$, one has
	\begin{equation}
	\label{convmesemp}
	 \widetilde{P}_n \underset{n\rightarrow +\infty}{\Rightarrow} P \qquad \mbox{and} \qquad \widehat{P}_n \underset{n\rightarrow +\infty}{\Rightarrow} C\, f^{\frac{d}{p+d}}(u) du.
	\end{equation}
	%where $\ds C=\left(\int_{\R} f^{\frac{d}{d+p}}(u) du \right) ^{-1}$.
\end{thm}
%\begin{thm}(Weighted empirical measure)
%	Let $\Gamma_n$ be an asymptotically optimal $n$-quantizer of $P$. Then, the corresponding empirical weighted measure converges weakly towards $P$
%	\begin{equation}
%	\label{convmespoids}
%	\ds \widetilde{P}_n \Rightarrow P \qquad \mbox{  when  } \, n\rightarrow +\infty.
%	\end{equation}
%\end{thm}
%\begin{thm}(Standard empirical measure)
%	Assume $P$ is absolutely continuous w.r.t the Lebesgue measure on $\R^d$ with density $f$. Then, the empirical standard measure associated to an asymptotically optimal $n$-quantizer $\Gamma_n$ of $P$ satisfy
%	\begin{equation}
%	\label{convmesstandard}
%	\ds \widehat{P}_n \Rightarrow C\, f^{\frac{d}{p+d}}(u) du \qquad \mbox{  when  } \, n\rightarrow +\infty,
%	\end{equation}
%	where $\ds C=\left(\int_{\R} f^{\frac{d}{d+p}}(u) du \right) ^{-1}$.
%\end{thm}
We hope to obtain such results for greedy sequences or, at least, for sub-optimal greedy sequences defined in the previous section. To this end, we {\em ``divide''} the two limits mentioned in~$(\ref{convmesemp})$, along the sequence $(W_i(a^{(n)}))_{1 \leq i \leq n}$ of the Vorono\"i weights associated to $(a_n)_{n \geq 1}$, and we obtain 
%\begin{equation*}
$n P_l(W_i(a^{(n)}))  \simeq \frac{1}{C} f^{1-\frac{d}{d+p}}(a_i^{(n)})
= C^{-1} f^{\frac{p}{d+p}}(a_i^{(n)}). $ 
%\end{equation*}
Hence, for all $i \in \{1, \ldots , n\}$, the limiting measure of the Vorono\"i cells of the greedy sequence is given by 
\begin{equation}
\label{poidslimite}
P_l(W_i(a^{(n)})) \simeq \frac{f^{\frac{p}{d+p}}(a_i^{(n)})}{C n}
\end{equation}
%where $C=\left(\int_{\R} f^{\frac{d}{d+p}}(u) du \right) ^{-1}$.\\
In other words, if the greedy sequences satisfy the convergence of the empirical measures, then the weights of the Vorono\"i cells, computed by
\begin{equation}
\label{poidsnormal}
\ds P(W_i(a^{(n)}))= F_{P}\left(a_{i+\frac{1}{2}}^{(n)}\right)-F_{P}\left(a_{i-\frac{1}{2}}^{(n)}\right),
\end{equation}
where $ a_{i+\frac{1}{2}}^{(n)}=\frac{a_i^{(n)}+a_{i+1}^{(n)}}{2}$ for $1 \leq i \leq n-1$ and $F_{P}$ is the c.d.f of $P$, must converge to the limit weights $P_l(W_i(a^{(N)}))$ given in~$(\ref{poidslimite})$. 

Numerical experiments were established for the normal distribution $\mathcal{N}(0,1)$, uniform distribution $\mathcal{U}([0,1])$, exponential distribution $\mathcal{E}(1)$ and Laplace distribution $\mathcal{L}(0,1)$. We observe positive results in the four cases, the weights of the Vorono\"i cells computed online get closer to the limit weights $P_l$ when $n$ increases. For the gaussian distribution, we observe a more important convergence for the subsequences $a^{(2^k-1)}$ (as predicted). We present, in figure~\ref{convempiricalmeasure} the obtained results for the exponential distribution where we compare the weights~$(\ref{poidsnormal})$ (blue) and the limit weights~$(\ref{poidslimite})$ (red) for different number of points $n$.
\begin{figure}
	\begin{center}
		%\vspace{-4cm}
		\begin{tabular}{cc}
			\includegraphics[width=.35\textwidth,height=.25\textheight,angle=0]{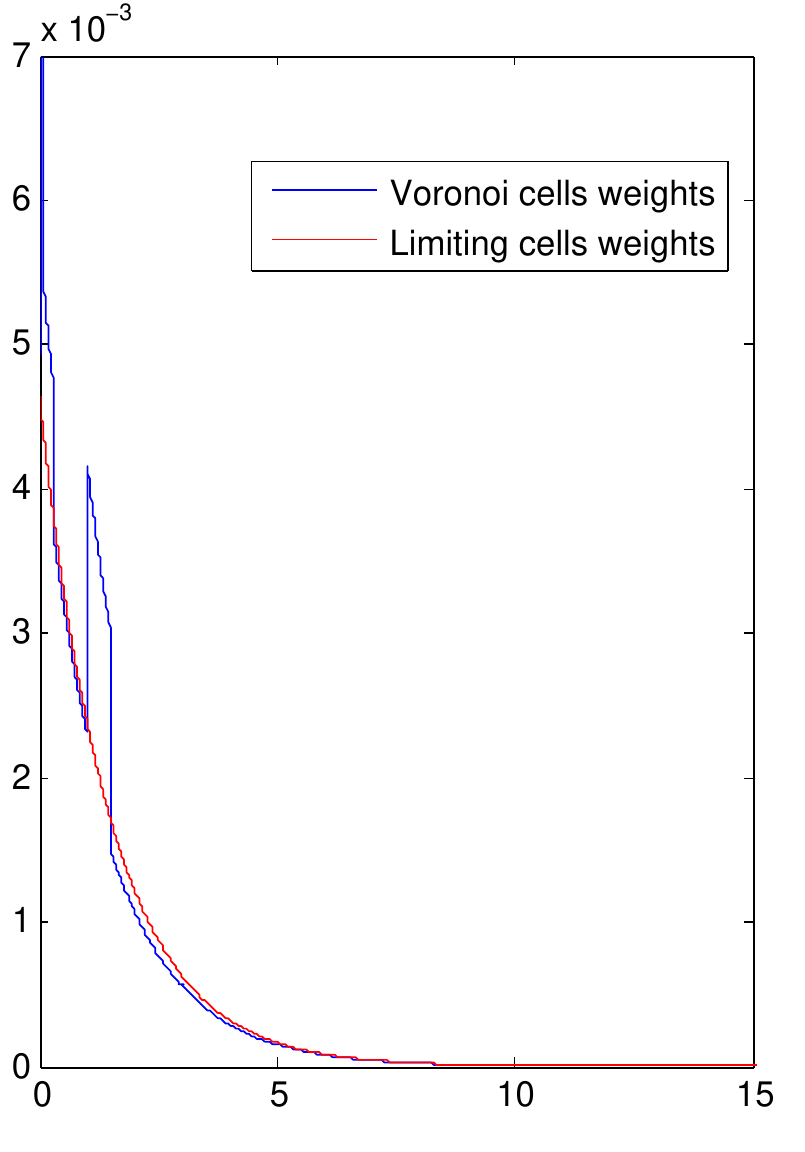} &
			\hspace{1cm}
			\includegraphics[width=.35\textwidth,height=.25\textheight,angle=0]{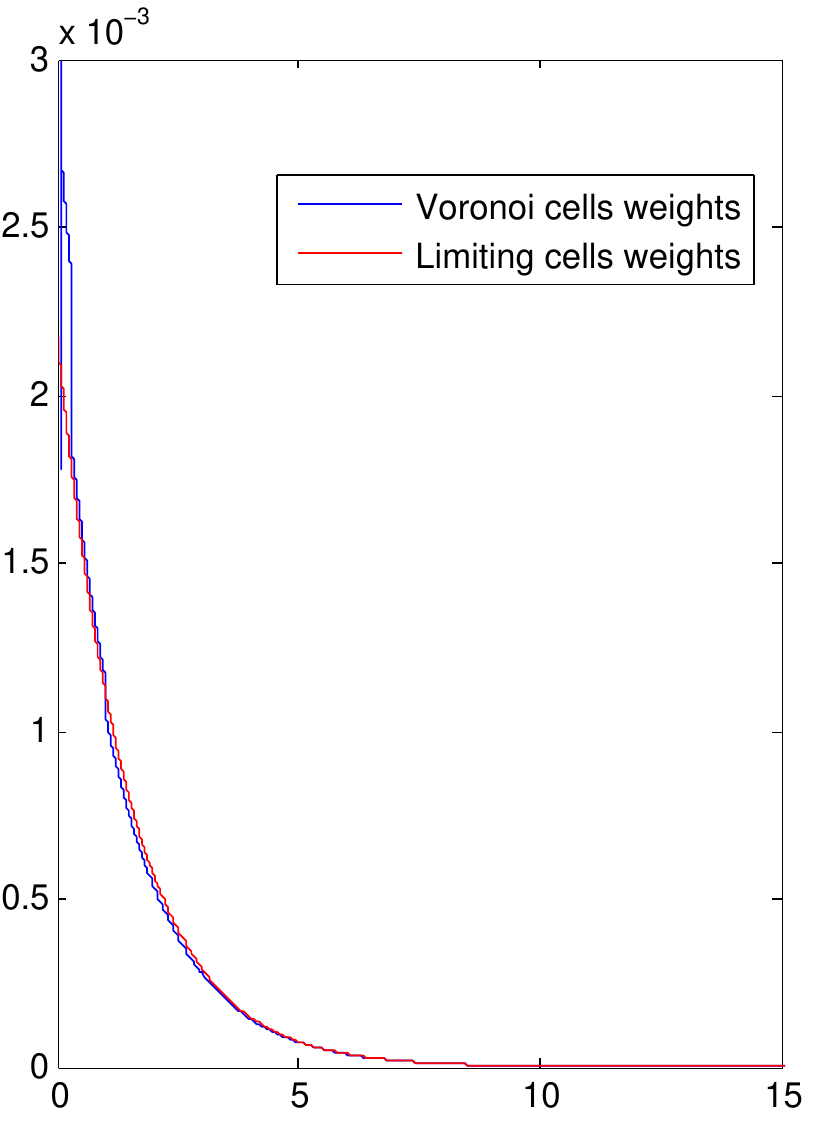}
		\end{tabular}
	\end{center} 
	\vspace{-0.5cm}
	\caption{Comparison of the weights of the Vorono\"i cells computed online (blue) to the limit weights of the cells (red) for the exponential distribution $\mathcal{E}(1)$ for $n=645$ (left) and $n=1\,379$ (right).}
	\label{convempiricalmeasure}
\end{figure}
\subsection{Stationarity and $\rho$-quasistationarity}
An interesting question is to see if the greedy sequences are stationary i.e. satisfy
$$a_i^{(n)}=\mathbb{E}(X|X \in W_i(a^{(n)})), i=1, \ldots,n,$$
or can be close to stationarity, a property shared by quadratic optimal quantizers. We compute the error $ \|\hat{X}^{a^{(n)}} -\E(X|\hat{X}^{a^{(n)}})\|_1$ under the standard empirical measure $ \widehat{P}_n= \frac{1}{n} \sum_{i=1}^{n} \delta_{a_i^{(n)}}$, i.e.
\begin{equation}
\label{aqs}
\sum_{i=1}^{n}\left| a_i^{(n)} - \mathbb{E}(X|X \in W_i(a^{(n)}))\right|= \sum_{i=1}^{n}\left| a_i^{(n)}-\ds\frac{\int_{W_i(a^{(n)})}\xi dP(\xi)}{P(W_i(a{(^n)}))}\right|.
\end{equation}
and we hope to observe a convergence to $0$, when $n$ increases. But, numerical experiments, conducted for several probability distributions, show that $a^{(n)}$ cannot be stationary, in the sense of~$(\ref{aqs})$. In fact, we will show, for a specific class of distributions, that greedy quantization sequences are not stationary (in the sense of~$(\ref{aqs})$), except when $n\in \{1,3\}$. We will use the following result given in \cite{Kieffer}.
\begin{thm} \label{kieffer} (J.C. Kieffer)
	Let $d=1$ and $\mu$ a probability distribution with log-concave density. Then, there exists a unique stationary quantizer of $\mu$.
\end{thm}

\begin{prop}
	\label{stationary}
	Let $X$ be a random variable with distribution $P$ which is symmetric and unimodal (log-concave density) and $a^{(n)}$ a corresponding greedy quantization sequence. Then, for every $n \in \mathbb{N}\setminus \{1,3\}$, the sequence $a^{(n)}$ is not stationary. 
\end{prop}
\begin{proof}   
	We suppose that $\mathbb{E}[X]=0$ (symmetric around $0$). If it is not the case, a translation gives the same results. We will detail the proof in $3$ cases\\ 
	$\rhd$ {\bf For $n=3$:} 
		Since $\mathbb{E}[X]=0$, the first point is $a_1=0$. A second point is given by 
		$$a_2 =  {\rm argmin}_{a \in \mathbb{R}} \mathbb{E}X^2 \wedge (X-a)^2 =\big\{ \nabla_{a_2} \,  \mathbb{E}X^2 \wedge (X-a_2)^2=\int_{W_2(a^{(i)})}(\xi - a_2)dP(\xi)=0\big\}.$$
		Hence, $a_2=\frac{\int_{W_2(a^{(n)})}\xi dP(\xi)}{P(W_2(a^{(n)}))}$ is stationary. The third point is $a_3=-a_2$ by symmetry of $P$ so $a_3$ is also stationary. Finally, $a_1=0$ is also stationary since
		$\int_{a_2/2}^{a_3/2}\xi dP(\xi)=\int_{a_2/2}^{-a_2/2}\xi dP(\xi)=0$. Consequently, the sequence $a^{(3)}=\{-a_2;a_1;a_2\}$ is stationary.\\% and, by theorem~$\ref{kieffer}$, $a^{(3)}$ is equal to the $3$-optimal quantizer of $P$.\\
		$\rhd$ {\bf For $n=2k$ even:}
		Since $P$ is unimodal, the stationary quantizer is unique, let $x^{(n)}$ be this quantizer, which is the $n$-optimal quantizer of $P$ because we know it is stationary. The symmetry of $P$ lets us know that the quantizer $(x_{n+1-l}^{(n)})_{1 \leq l \leq n}$ of $P$ is also stationary, so, for every $l \in \{1, \ldots,n\}$, $x_l^{(n)}+x_{n+1-l}^{(n)}=0.$ 
		Since $n=2k$ is even, we have, in particular, 
		$$x_{k}^{(n)}=-x_{n+1-k}^{(n)}=-x^{(n)}_{n+1-\frac{n}{2}}=-x^{(n)}_{k+1},$$
		so,
		$x_k^{(n)}<0<x_{k+1}^{(n)}$
		and, since, $x_k^{(n)}$ et $x_{k+1}^{(n)}$ are two consecutive terms of the grid, we deduce that $0$ is not an element of $x^{(n)}$, and hence can not be a point of a stationary quantizer. Consequently, the greedy sequence starting at $a_1=0$ can not be stationary.\\		
		$\rhd$ { \bf For $n=2k+1$ odd:}
		We consider the greedy non-stationary sequence  $a^{(2k)}$ of even size. There exists, at least, one non-stationary Vorono\"i cell $W_i(a^{(2k)})$. Its symmetric cell $W_{2k+1-i}(a^{(2k)})$ will also be non-stationary because $P$ is symmetric. So, we have at least $2$ non stationary Vorono\"i cells. While building the sequence $a^{(2k+1)}$, we add a new point which will be in one of the Vorono\"i cells without modifying the others. If the new point is added in one of the non-stationary cells, we know that the second one will remain untouched, having, at least, one non-stationary cell in $a^{(2k+1)}$. And, if the new point is not in these cells, then they will remain untouched and there will be, at least, $2$ non-stationary cells in $a^{(2k+1)}$.		\hfill $\square$\\
\end{proof}

However, further different numerical observations show that most greedy quantization sequences satisfy a certain criteria that we can call $\rho$-quasi-stationarity. This criterion approaches to the stationary character verified by optimal quantizers and  can be defined, for $r \in \{1,2\}$ and $\rho \in [0,1]$, by

\begin{equation}
\label{rhoaqs}
 \|\widehat{X}^{a^{(n)}} -\mathbb{E}(X|\widehat{X}^{a^{(n)}})\|_r= o(\| \widehat{X}^{a^{(n)}}-X\|_{1+\rho}^{1+\rho}), \quad \mbox{or} \quad \frac{ \|\widehat{X}^{a^{(n)}} -\mathbb{E}(X|\widehat{X}^{a^{(n)}})\|_r}{\| \widehat{X}^{a^{(n)}}-X\|_{1+\rho}^{1+\rho}} \underset{n \rightarrow +\infty}{\longrightarrow} 0.
 \end{equation}
It is satisfied by greedy sequences for $\rho$ lower than certain optimal values $\rho_l$ depending on the distribution $P$ and on the value of $r$. We expose, in table~$\ref{rhooptimal}$, these values of $\rho_l$ for $r=1$ and $r=2$ for the normal, uniform or Gaussian distribution.
\begin{table}
	\begin{center}
		\begin{tabular}{|l|c|c|r|}
			\hline
			& $\mathcal{N}(0,1)$ & $\mathcal{U}([0,1])$ & $\mathcal{E}(1)$ \\ 
			\hline
			$r=1$    &      $\rho_l= 0.92$   &     $\ds \rho_l=\frac{3}{4}$   & $\ds \rho_l=\frac{2}{3} $ \\
			\hline
			$r=2$     &  $\rho_l=0.47$       &     $\ds \rho_l=\frac{3}{8}$   &  $\ds \rho_l=\frac{1}{3}$\\
			\hline
		\end{tabular}
		%\vspace{-0.5cm}
	\end{center}
	\caption{Values of optimal $\rho_l$ for different distributions and for $r \in \{1;2\}$.} 
	\label{rhooptimal}
\end{table}
This property is important because it brings improvements to quantization-based numerical integration. The error induced by this integration using standard cubature formula for functions with $\rho$-H\"older gradient is bounded by 
$$|\mathbb{E}f(X)-\mathbb{E}f(\widehat{X}^{a^{(n)}})|  \leq \frac{1}{1+\rho}[\nabla f]_{\rho} \|X-\widehat{X}^{a^{(n)}}\|_{1+\rho}^{1+\rho}$$
while the classical error bound is given by (see \cite{Pages15})
$$|\E f(X)-\E f(\widehat{X}^{a^{(n)}})| \leq [f]_{\text{Lip}} \|X-\widehat{X}^{a^{(n)}}\|_p.$$
In fact, if $\rho \in [0,1]$ and $f$ is a continuous function with $\rho$-h\"older gradient with H\"older coefficient $[\nabla  f]_{\rho}$, we have 
\begin{align*}
\mathbb{E}f(X)-\mathbb{E}f(\widehat{X}^{a^{(n)}}) \leq & \,  \mathbb{E}\big(\nabla f(\widehat{X}^{a^{(n)}})|X-\widehat{X}^{a^{(n)}}\big) \\
& + \E\left[  \int_0^1 \left( \nabla f\left( \widehat{X}^{a^{(n)}}+t(X-\widehat{X}^{a^{(n)}})\right)-\nabla f(\widehat{X}^{a^{(n)}})| X-\widehat{X}^{a^{(n)}}\right)dt \right].
\end{align*}
Since,
\begin{align*}
\int_0^1 \left( \nabla f\left(  \widehat{X}^{a^{(n)}}+t(X-\widehat{X}^{a^{(n)}})\right)-\nabla f( \widehat{X}^{a^{(n)}})| X-\widehat{X}^{a^{(n)}}\right)dt 
\leq [\nabla f]_{\rho} \mathbb{E}|X-\widehat{X}^{a^{(n)}}|^{1+\rho}\int_0^1 t^{1+\rho} dt, 
\end{align*}
and 
\begin{align*}
%\label{etoile}
\mathbb{E}(\nabla f(\widehat{X}^{a^{(n)}})|X-\widehat{X}^{a^{(n)}})  =\mathbb{E}(\nabla f(\widehat{X}^{a^{(n)}})|X)-\mathbb{E}(\nabla f(\widehat{X}^{a^{(n)}})|\widehat{X}^{a^{(n)}}) \nonumber = \mathbb{E}\left( \nabla f(\widehat{X}^{a^{(n)}})|\mathbb{E}(X|\widehat{X}^{a^{(n)}})-\widehat{X}^{a^{(n)}}\right),
\end{align*}
 we have
\begin{equation*}
|\mathbb{E}f(X)-\mathbb{E}f(\widehat{X}^{a^{(n)}})| \leq \|\nabla f(\widehat{X}^{a^{(n)}})\|_2 \|\mathbb{E}(X|\widehat{X}^{a^{(n)}})-\widehat{X}^{a^{(n)}}\|_2 +  \frac{1}{1+\rho}[\nabla f]_{\rho} \|X-\widehat{X}^{a^{(n)}}\|_{1+\rho}^{1+\rho} .
\end{equation*}
Hence, if~$(\ref{rhoaqs})$ is satisfied, then on can conclude that
\begin{equation}
\label{bornerho}
\limsup_{n} \frac{|\mathbb{E}f(X)-\mathbb{E}f(\widehat{X}^{a^{(n)}})|}{ \|X-\widehat{X}^{a^{(n)}}\|_{1+\rho}^{1+\rho}} \leq \frac{1}{1+\rho}[\nabla f]_{\rho}
\end{equation}
and hence the gain in the quantization-based numerical integration error bounds.
\subsection{Discrepancy of greedy sequences}
The comparison established, in the beginning of section~\ref{algorithmics}, between greedy quantization-based numerical integration and quasi-Monte Carlo methods, showing a gain of $\log(n)$-factor with greedy quantization in terms of convergence rate, drives us to build a relation, based on Pro\"inov's Theorem~$\ref{proinov}$, between the error quantization and the discrepancy. In fact, for every $n$-tuple $\Xi=(\xi_1,\ldots,\xi_n) \in [0,1]^n$, noticing that a Lipschitz function $f$ has always a finite variation 
%since $V(f)= \sup_{\sigma=(t+0=0<t_1<\ldots<t_n=1)} \sum_{\sigma} |f(t_{i+1})-f(t_i)| \leq [f]_{Lip} \times (1-0)=[f]_{Lip},$ 
and considering the function $f:u \rightarrow \min_{1 \leq i \leq n}|u-\xi_i|$ which is $1$-Lipschitz (since $\left| \min_i a_i -\min_i b_i  \right| \leq \max_i |a_i-b_i|)$ and satisfies $f(\xi_i)=0$ for every $i \in \{1, \ldots ,n\}$ and $\ds \int_0^1 f(u)du =e_1(X,\mathcal{U}([0,1]))$, one applies the Koksma-Hlawka inequality ($\ref{koksmahlawka}$) to $f$ to deduce that
\begin{equation}
\label{lien1}
e_1(\Xi,\mathcal{U}([0,1])) \leq D_n^*(\Xi).
\end{equation}
%\begin{prop}
%	For every $n$-tuple $\Xi=(\xi_1,\ldots,\xi_n) \in [0,1]^n$, one has 
%	\begin{equation}
%	\label{lien1}
%	e_1(\Xi,\mathcal{U}([0,1])) \leq D_n^*(\Xi).
%	\end{equation}
%\end{prop}
%\begin{proof}
%	We start by noticing that any Lipschitz function $f$ has a finite variation since
%	$V(f)= \sup_{\sigma=(t+0=0<t_1<\ldots<t_n=1)} \sum_{\sigma} |f(t_{i+1})-f(t_i)| \leq [f]_{Lip} \times (1-0)=[f]_{Lip}.$
%	Then, we consider the function $f$ defined by $f(u)= d(u, X)=\min_{1 \leq i \leq n}|u-\xi_i|$ which is $1$-Lipschitz, owing to the fact that $\left| \min_i a_i -\min_i b_i  \right| \leq \max_i |a_i-b_i|.$ Furthermore, having $f(\xi_i)=0$ for every $i \in \{1, \ldots ,n\}$ and $\ds \int_0^1 f(u)du =e_1(X,\mathcal{U}([0,1]))$, one deduces the result by applying the Koksma-Hlawka inequality ($\ref{koksmahlawka}$) to this particular $f$.
%	\hfill $\square$
%\end{proof}
%This relation between the star discrepancy of a sequence $\Xi$ on $[0,1]$ and the quantization error induced by this sequence with respect to the uniform distribution emphasizes the gain of the $\log(n)$-factor with greedy quantization. 
This motivates us to study the discrepancy of greedy sequences hoping that they can be comparable to low discrepancy sequences. We compute the discrepancy of greedy quantization sequences, for $d \in \{1,2,3\}$, using formulas given in \cite{Doerr14} (theorems $1$, $2$ and $3$) and recalled below.
\begin{thm}
	$(a)$ Let $\Xi=(\xi_i)_{1 \leq i \leq n}$ be a sequence in $[0,1]$ and assume $\xi_1 \leq \ldots \leq \xi_n$. Then, for every $n \geq 1$
		\begin{equation}
		D_n^*(\Xi)=\max_{1 \leq i \leq n} \max \left\{\frac in -\xi_i\, ,\, \xi_i -\frac{i-1}{n} \right\}=\frac{1}{2n} +\max_{1 \leq i \leq n} \left|\xi_i-\frac{2i-1}{2n} \right|.
		\end{equation}
	$(b)$ Let $\Xi=(\xi_i)_{1 \leq i \leq n}$ be a sequence in $[0,1]^2$ where each $\xi_i$ have two components $(x^1_i,x_i^2)$. Assume $x_1^1 \leq \ldots \leq x_n^1$. For every $i \in \{1,\ldots,n\}$, we consider $(\xi_{i,0},\ldots,\xi_{i,i+1})$ an increasing reordering of $(0,x_1^2,\ldots,x_i^2,1)$. Then,
		\begin{equation}
		D_n^*(\Xi)=\max_{1 \leq i \leq n} \max_{0 \leq k \leq i} \max \left\{\frac kn -x^1_i \xi_{i,k}\, ,\, x^1_{i+1}\xi_{i,k+1} -\frac{k}{n} \right\}.
		\end{equation}
	$(c)$ Let $\Xi=(\xi_i)_{1 \leq i \leq n}$ be a sequence in $[0,1]^3$ where each $\xi_i$ have three components $(x^1_i,x_i^2,x_i^3)$. Assume $x_1^1 \leq \ldots \leq x_n^1$. For every $i \in \{1,\ldots,n\}$, we consider $(\xi_{i,0},\ldots,\xi_{i,i+1})$ an increasing reordering of $(0,x_1^2,\ldots,x_i^2,1)$.	For a fixed $i$ and $k \in \{1,\ldots,i\}$, we consider $\{\eta_{i,k,0},\ldots,\eta_{i,k,k+1}\}$ an increasing reordering of $\{0,x_1^3,\ldots,x^3_k,1\}$. Then,
		\begin{equation}
		D_n^*(\Xi)=\max_{1 \leq i \leq n} \max_{0 \leq k \leq i} \max_{0 \leq l \leq k} \max \left\{\frac ln -x^1_i \xi_{i,k} \eta_{i,k,l}\, ,\, x^1_{i+1}\xi_{i,k+1} \eta_{i,k,l+1} -\frac{l}{n} \right\}.
		\end{equation}
\end{thm}
Numerical results show that, in the one-dimensional case, greedy sequences can be used as a low discrepancy sequence. But, when $d$ becomes larger than $1$, the situation becomes less convincing. In fact, if we use pure greedy sequences designed by implementing Lloyd's algorithm, the discrepancy of these sequences and that of low discrepancy sequences (Niederreiter sequences for example) are comparable and the results are not so bad, but the problem that arises is the complexity of the computations
%and the cost caused by computing these sequences and their weights (that are not obtained by closed formulas) and,as a result, 
making greedy sequences less practical. On the other hand, if we use the greedy product multidimensional grids to solve this problems, the computation will be less expensive but the numerical experiments show that there is no gain in terms of discrepancy. Figure~$\ref{discrepancy}$ shows a comparison of the discrepancy of a Niederreiter sequence in dimension $2$ to that of a product greedy quantization sequence of $\mathcal{U}([0,1]^2)$ on the one hand, and to that of pure greedy quantization sequence of $\mathcal{U}([0,1]^2)$ implemented by Lloyd's algorithm on the other hand, emphasizing the conclusions made. \\

\begin{figure}
	\begin{center}
		\begin{tabular}{cc}
			\includegraphics[width=8cm,height=5.8cm,angle=0]{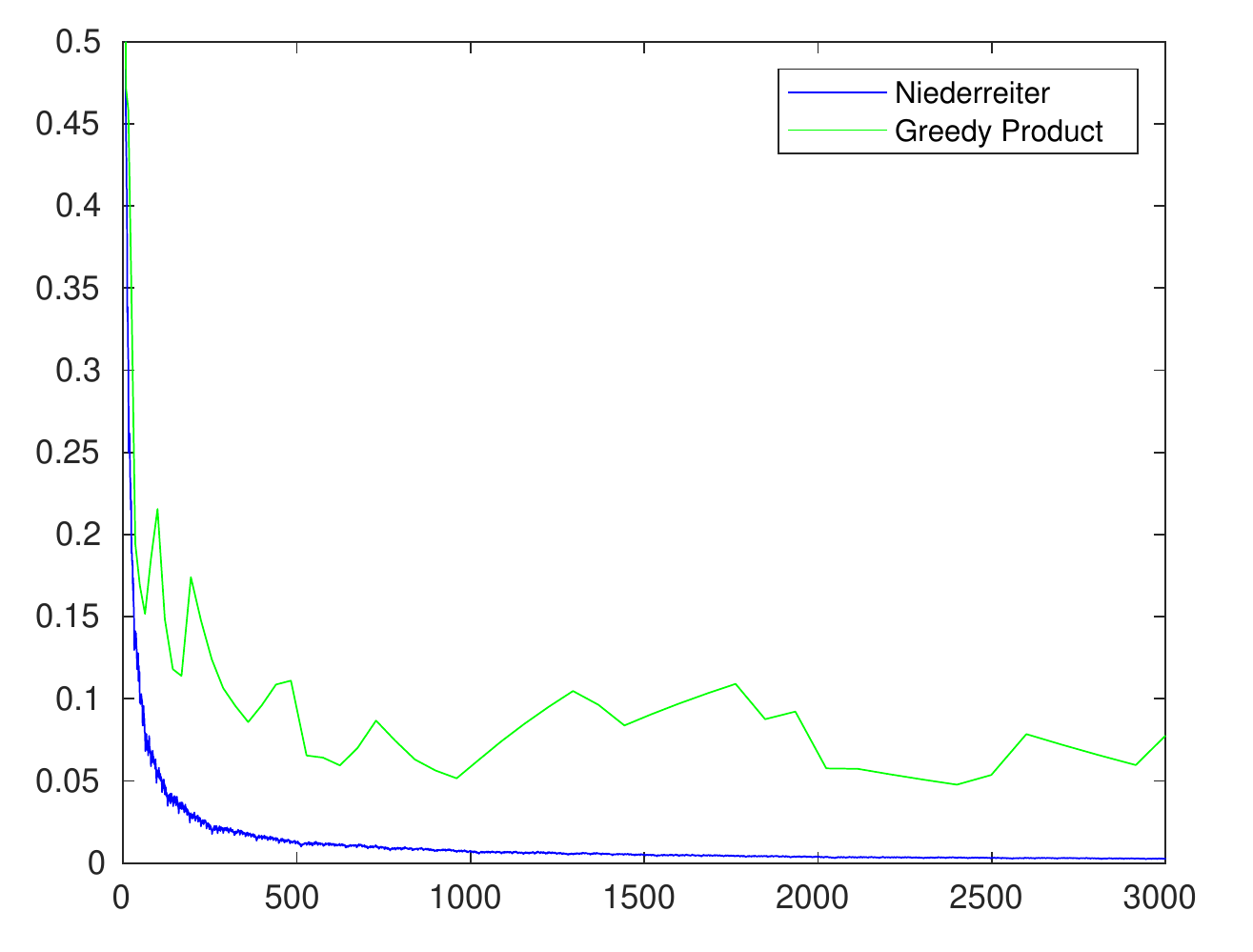} 
			\includegraphics[width=7.7cm,height=5.7cm,angle=0]{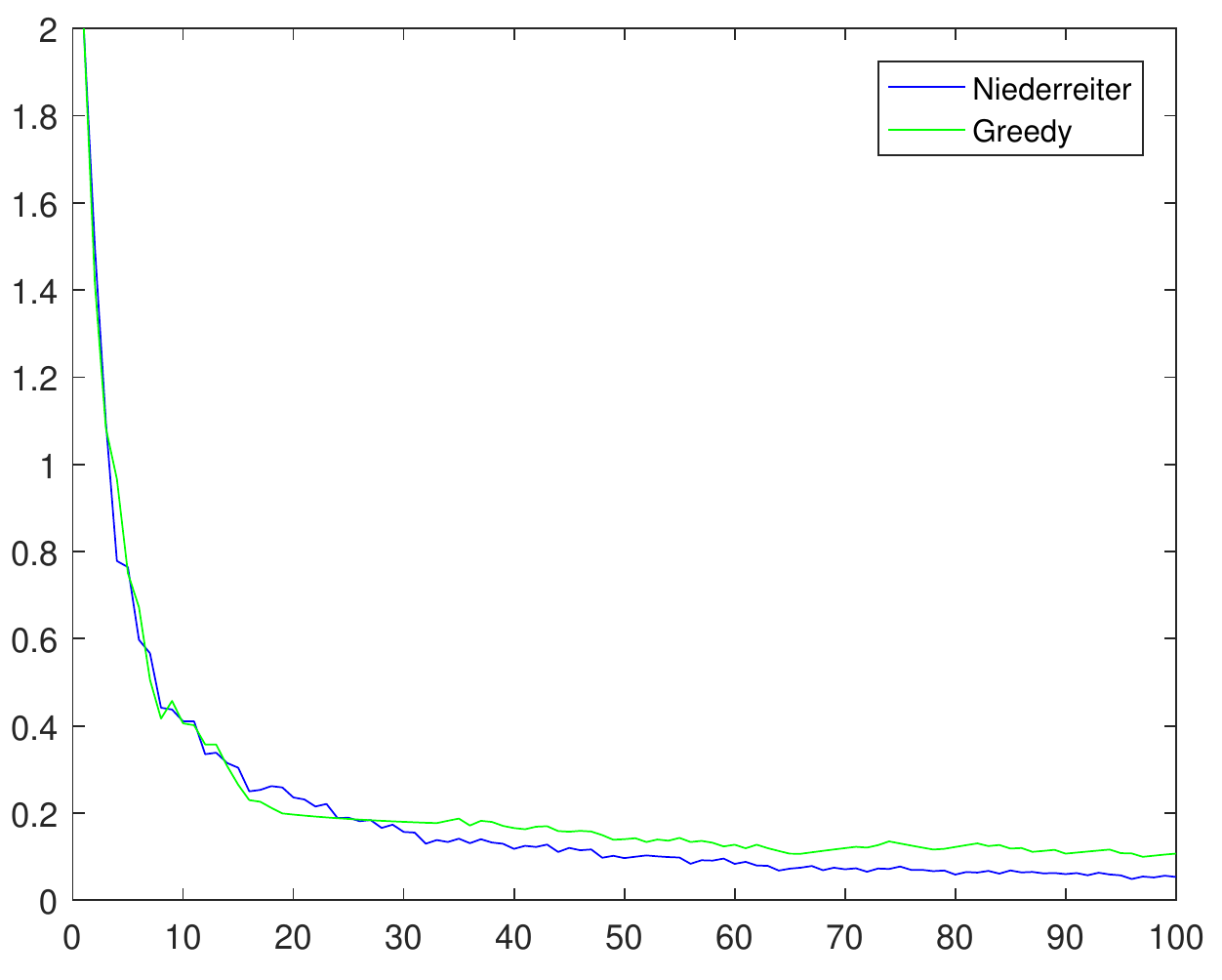}
		\end{tabular}
	\end{center}
	\vspace{-0.5cm}
	\caption{Comparisons of the star discrepancy of the Niederreiter sequence to a greedy product quantization sequence of the uniform distribution $\mathcal{U}([0,1]^2)$ (left) and to a pure greedy quantization sequence (right) for $d=2$.}
	\label{discrepancy}
\end{figure}
The positive results obtained in the one-dimensional case encourage us to try and manipulate low discrepancy sequences, such as Van der Corput sequences, in order to be able to use them as greedy quantization sequences. In other words, we will assign to them a Vorono\"i diagram, compute the weights of the corresponding Vorono\"i cells instead of considering uniform weights and observe the impact the may bring to numerical integration. To this end, we consider a basic example where we compute the price of a European call $$C_0=\E[(X_T-K)_{+}]$$ for a maturity $T$ and a strike price $K$ where the price of the asset $X_t$ at a time $t$ is given by
$$X_t=x_0e^{(r-\frac{\sigma^2}{2})t +\sigma \sqrt{t} Z_t}$$
where $r$ is the interest rate, $\sigma$ the volatility and $(Z_t)_{0 \leq t \leq T}$ is an i.i.d. sequence of random variables with distribution $\mathcal{N}(0,1)$. We compute the price of this European call via a classical quadrature formula using the new weights $p_i^n$ assigned to the VdC sequence instead of uniform weights. We consider 
$$T=1 , \, K=9, \, x_0=10, \, \mu=0.06, \, \sigma =0.1.$$
The exact price is approximately equal to $1.5429$ due to the closed formula known in the Black-Scholes case. In figure~\ref{intnumVdC}, we compute the error induced by this approximation and we compare it to the one obtained by a classical quasi-Monte Carlo method (i.e. where we use the uniform weights of a VdC sequence) and to the one obtained by a quantization-based numerical integration quadrature formula using a greedy quantization sequence of the $\mathcal{U}([0,1])$-distribution. We observe that the procedure using the greedy quantization sequence converges faster than the ones using the Van der Corput sequence. Consequently, one can say that greedy sequences are more advantageous than low discrepancy sequences, even if we assign to them non-uniform weights.  \\

\begin{figure}
	\begin{center}
	\vspace{-0.5cm}
		\includegraphics[width=11cm,height=7cm,angle=0]{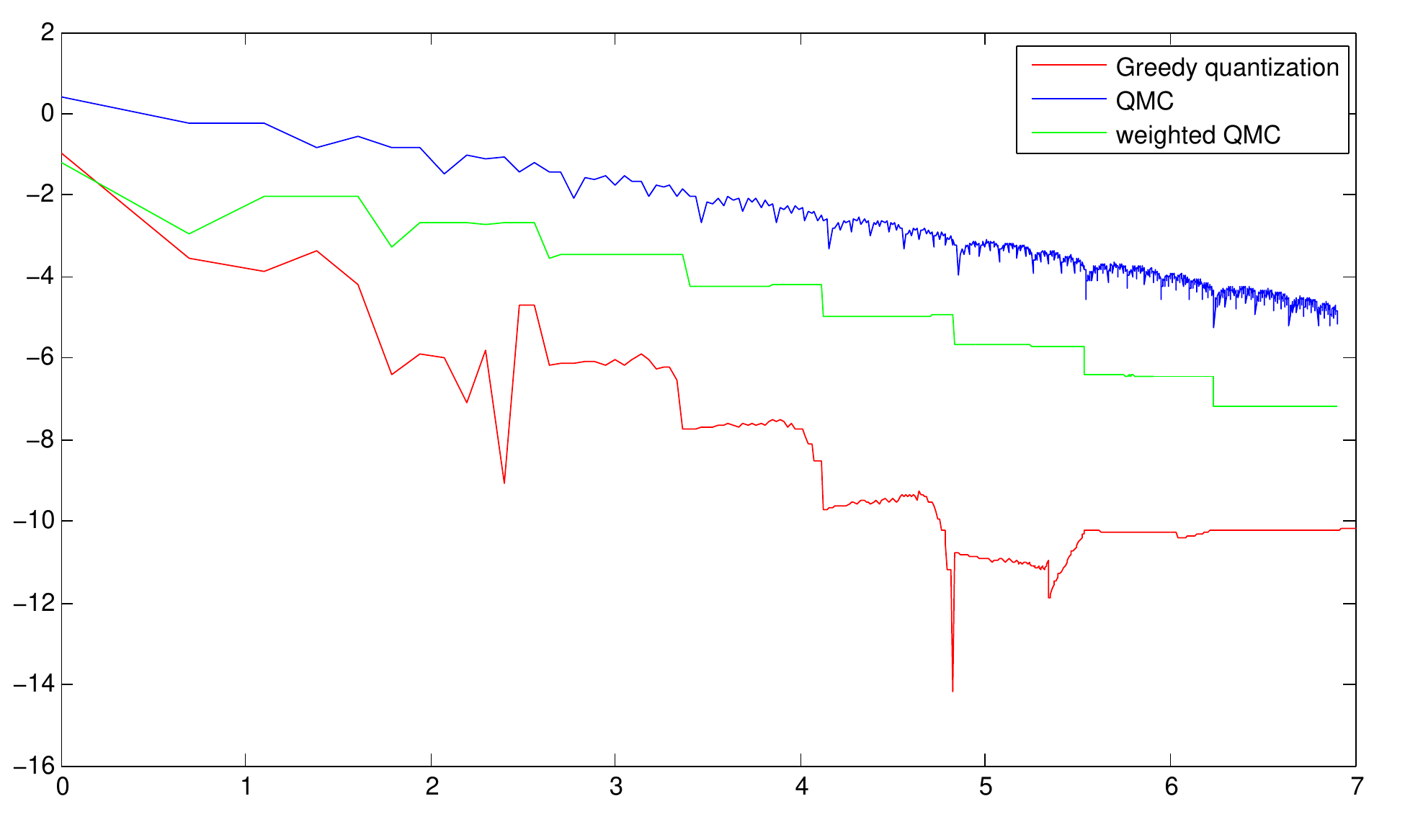}
	\end{center} 
	\vspace{-1cm}
	\caption{Price of a European call in a Black-Scholes model via a usual QMC method (blue), greedy quantization-based quadrature formula (red) and quadrature formula using VdC sequence with non-uniform weights (logarithmic scale).}
	\label{intnumVdC}
\end{figure}
\smallskip
\small
\noindent {\em Acknowledgments.} The authors would like to express a sincere gratitude to, Dr. Rami El Haddad, the co-advisor of R. El Nmeir, for his help and advice during this work. Also, they would like to acknowledge the National Council for Scientific Research of Lebanon (CNRS-L) for granting a doctoral fellowship to Rancy El Nmeir, in a joint program with Agence Universitaire de la Francophonie of the Middle East and the research council of Saint-Joseph University of Beirut.

\end{document}